%% file: ex_article.tex
\documentclass[onefignum,onetabnum]{siamonline250211}
\usepackage{amssymb,amsfonts}
\usepackage{enumerate}
\usepackage{mathtools}
\usepackage{dsfont}
\usepackage{pgfplots}
\usepackage{enumitem}
\usepackage{graphicx}
\usepackage{ulem}
\usepackage{xcolor}
\usetikzlibrary{arrows.meta}
\definecolor{darkgreen}{rgb}{0.0, 0.5, 0.0} 
\crefname{equation}{}{}
\Crefname{equation}{}{}
\crefname{figure}{Figure}{Figures}
\Crefname{figure}{Figure}{Figures}
\crefname{section}{section}{sections}
\Crefname{section}{Section}{Sections}
\crefname{assumption}{Assumption}{Assumptions}
\crefname{definition}{Definition}{Definitions}
\crefname{theorem}{Theorem}{Theorems}
\crefname{proposition}{Proposition}{Propositions}
\crefname{corollary}{Corollary}{Corollaries}
\crefname{lemma}{Lemma}{Lemmas}
\crefname{problem}{Problem Setup}{Problem Setups}
\usepackage{todonotes}
\usetikzlibrary{calc}

\input{definition}

\newtheorem{assumption}{Assumption}

\DeclareMathOperator{\sech}{sech}
\def\arrowlen{1.2}

\allowdisplaybreaks[4]

\input{ex_shared}

\ifpdf
\hypersetup{
  pdftitle={Analysis of the long-term behavior of the ``Bando--follow-the-leader'' car-following model},
  pdfauthor={Fei Cao, Xiaoqian Gong, and Alexander Keimer}
}
\fi

\begin{document}

\title{Analysis of the long-term behavior of the ``Bando--follow-the-leader'' car-following model}

\author{Fei Cao\thanks{Department of Mathematics, Amherst College, 
220 S Pleasant St, Amherst, MA 01002, USA 
(fcao@amherst.edu, xgong@amherst.edu).}
\and Xiaoqian Gong\footnotemark[1]
\and Alexander Keimer\thanks{Department of Mathematics, University of Rostock, 
Ulmenstraße 69, 18057 Rostock, Germany 
(alexander.keimer@uni-rostock.de).}}

\maketitle

\begin{abstract}
In this article, we investigate the long-term behavior of the ``Bando--follow-the-leader'' car-following model, whose well-posedness and stability with respect to delay were analyzed in a recent work \cite{gong2023well}. We first establish the collision-free property of the model with \(N+1\in\N_{\geq2}\) vehicles over an infinite time horizon, assuming that the trajectory of the first vehicle is prescribed, by demonstrating the existence of a uniform strictly positive lower bound on the space headway between adjacent vehicles. 
Furthermore, assuming that the first vehicle travels at a constant velocity and \(N\in\N_{\geq1}\) vehicles follow it according to the Bando--follow-the-leader model on a single lane, our main results state that, with certain reasonable constraints imposed on the modeling parameters, all \(N\) following vehicles will eventually (i.e., when time goes to infinity) converge to the same headway and velocity with a globally exponential convergence rate. The analytical methods are based on Lyapunov functions and a perturbation argument. 
Numerical simulations are also provided to illustrate the obtained theoretical convergence guarantees.
\end{abstract}



\section{Introduction}
\label{sec:sec1}
\setcounter{equation}{0}

Mathematical traffic models can be classified into three types--\\
microscopic, mesoscopic, and macroscopic--depending on the scale at which they represent vehicular traffic. Comprehensive reviews of car-following models can be found in \cite{wang2023car, he2023microscopic, zhang2024car, ahmed2021review, fosureview}.

Microscopic car-following models describe the dynamics of individual vehicles in traffic flow. First-order models assume that drivers instantly adjust their speed to a desired value determined by the space headway, resulting in simpler dynamics without explicit acceleration. Second-order models assume that each driver adjusts their vehicle's acceleration in response to factors such as the space headway to the preceding vehicle, the driver’s desired velocity, and the relative velocity between adjacent vehicles, leading to more realistic but more complex behavior. Classical examples of car-following models include the optimal velocity model \cite{bando1995dynamical}, the follow-the-leader model \cite{gazis1961nonlinear}, the intelligent driver model \cite{treiber2000congested} and its variations \cite{albeaik2022limitations}, and the Gipps model \cite{gipps1981behavioural}. Rigorous analysis of the collision-free behavior of certain microscopic car-following models can also be found in \cite{gong2023well, 10552791}. To capture driver behavior more realistically, delayed car-following models have been developed by incorporating factors such as reaction delays, finite response times, and systematic sensor delays \cite{bando2000delay, yu2014new, gunter2020commercially, gong2023well}. 

Macroscopic traffic flow models treat traffic as a continuous fluid to describe aggregate features such as congestion, delay, and queue formation. These models are less computationally complex \cite{khan2019macroscopic}; examples include the Lighthill--Whitham--Richards model \cite{lighthill1961ii,richards1956shock}, the Payne--Whitham model \cite{payne1973freeway, whiteman1975linear}, two-phase models \cite{blandin2011general, colombo2003hyperbolic}, and the Aw--Rascle--Zhang model \cite{aw2000resurrection, zhang2002non, fan2017collapsed, chiarello2020micro}. Nonlocal macroscopic traffic flow models allow each vehicle to respond to conditions further downstream by considering an average traffic density over a region ahead, capturing the driver's anticipation of upcoming conditions \cite{keimer2017existence, karafyllis2022analysis, chiarello2020overview, aljamal2018comparison, chiarello2020non, bressan2020entropy}. These models are suited for large-scale, network-wide applications in which macroscopic characteristics such as speed, density, and flow are of primary interest.

Mesoscopic models bridge the gap between microscopic and macroscopic models by capturing the behavior of individual drivers in a probabilistic framework while maintaining a simplified depiction of driving dynamics \cite{de2019mesoscopic, gong2023mean,fornasier2014mean}. In these models, traffic is treated as a series of groups of vehicles or even individual vehicles, yet their evolution is governed by macroscopic flow laws. This approach preserves much of the computational efficiency of macroscopic models while providing a richer level of detail.

In this article, we study the well-posedness of a second-order microscopic car-following model---the Bando follow-the-leader (Bando-FtL) model introduced in \cite{stern2018dissipation}---on a single lane with (potentially) infinitely many vehicles, and we analyze its long-term behavior over an infinite time horizon. The Bando-FtL model combines two classical car-following models: the optimal velocity (Bando) model \cite{bando1995dynamical} and the follow-the-leader (FtL) model \cite{gazis1961nonlinear}. The Bando component captures the driver’s tendency to adjust toward an optimal velocity determined by the vehicle's space headway, while the FtL component captures their tendency to align the velocity of the vehicle with the one directly ahead.

The Bando-FtL model has been shown to reproduce stop-and-go waves on a ring road \cite{delle2019feedback}. Its well-posedness and that of a delayed extension were established in \cite{gong2023well}, including results on the existence of a strictly positive lower bound on vehicle headway and the nonnegativity of velocities over finite time horizons. A stochastic perturbation of the Bando-FtL model was analyzed in \cite{nick2022near}, whose authors proved collision-free behavior when the leading vehicle travels at a constant velocfity and showed that, for small noise levels, collisions become asymptotically unlikely over long time intervals. The linear stability of traffic comprising finitely many human-driven vehicles on a ring road, both without and with a single controlled vehicle, was examined in \cite{hayat2023dissipation} by linearizing around the steady state; the exponential stability of the steady state was proven under a proportional--integral control law. The classical stability and weak string stability of the Bando-FtL model on a ring road were investigated in \cite{giammarino2020traffic} by linearizing the dynamics around the equilibrium, both theoretically and numerically. In \cite{chou2024stability}, the authors developed an ODE model for vehicles on a ring road and evaluated its local stability using string stability criteria and an eigenvalue analysis of the linearized system around its equilibrium.

These studies naturally raise several questions: 
\begin{enumerate}
\item Does the Bando-FtL model remain well-posed when extended to an infinite number of human-driven vehicles over an unbounded time horizon? 
\item What is the steady state of the model for a finite number of vehicles traveling on a single lane?
\item How does this steady state extend as the number of vehicles tends to infinity?
\item What is the rate of convergence toward the steady state in both the finite- and infinite-vehicle settings? In particular, existing stability results are local, ensuring convergence only for initial conditions sufficiently close to equilibrium; can a global stability result be established?
\end{enumerate}

In this paper, we establish the existence of a strictly positive, uniform lower bound on each vehicle's space headway over an infinite time horizon for infinitely many vehicles under broad and general assumptions. Furthermore, we analyze the long-term behavior of the Bando-FtL model for finitely many vehicles by first proving the existence and uniqueness of an equilibrium and then employing classical energy methods, using suitable Lyapunov functions to demonstrate convergence to the steady state. We also provide a rigorous global estimate of the convergence rate under suitable parameter constraints. Via a perturbation argument, we extend the long-term behavior of the Bando-FtL model with finitely many vehicles to the case with infinitely many. 

The obtained well-posedness result for the Bando-FtL model with infinitely many vehicles on an infinite time horizon and the corresponding long-term behavior are also useful when investigating the equilibrium of macroscopic traffic oscillations \cite{nie2010equilibrium, yildirimoglu2014approximating}. Another related topic is the traveling wave solutions in traffic flow modeling \cite{ridder2018traveling, ikeda2024existence}.

The paper is organized as follows. In \cref{sec: model}, we describe the Bando-FtL model and review known results concerning its well-posedness over a finite time horizon. \Cref{sec: well_posedness} extends the well-posedness analysis to an infinite time horizon and establishes the existence and uniqueness of an equilibrium state. In \cref{sec:sec_two_vehicle}, we investigate the model's long-term behavior using energy methods for both finitely and infinitely many human-driven vehicles. \Cref{sec: conclusion} concludes the paper and outlines potential directions for future research.

\section{The Bando-FtL model}
\label{sec: model}
In this section, we recall the Bando-FtL model and its well-posedness over a finite time horizon. Let \(v_{\max} \in \R_{>0}\) be the maximal allowable velocity of each vehicle. We assume that the first vehicle's trajectory is predetermined and that the accelerations of the \(N\) following vehicles (hereafter, ``followers'') \(i \in \{2, \dots, N+1\}\) are governed by the Bando-FtL dynamics, which are defined as follows.

\begin{definition}[Bando-FtL acceleration]\label{defi:bando_acceleration}
Let \((\alpha,\beta)\in\R_{>0}^{2}\), and set
\[
A \coloneqq \big\{(a,b,c)\in\R^{3}:\ a>0\big\}
\]
and
\begin{equation}
    V\in C^{1}(\R_{\geq0};\R_{\geq0})\cap L^{\infty}(\R_{\geq0};\R_{\geq0})\ \text{ satisfying } V' \in (0, L] \text{ with } L>0.
    \label{eq:assumption_V}
\end{equation}
Then, the Bando-FtL acceleration is defined by the following function:
\begin{align}
    \textnormal{Acc}: \begin{cases}
         A&\rightarrow\R\\
    (a,b,c)&\mapsto  \alpha \left(V(a)-b\right)+\beta \tfrac{c-b}{a^{2}}.
        \end{cases}\label{defi:ACC}
\end{align}
We call \(V\) in \cref{defi:ACC} the optimal velocity function of the Bando-FtL acceleration.
\end{definition}
\begin{remark}[some comments on the Bando-FtL acceleration] Let \(v\), \(v_{\ell}\), and \(h\) be the follower's velocity, the leading vehicle's (hereafter, ``leader's'') velocity, and the follower's space headway, respectively. Given \((h, v, v_{\ell}) \in \mathbb{R}^3 \) with \(h>0\) and \((\alpha, \beta) \in \mathbb{R}_{>0}^2\), the Bando-FtL acceleration \(\textnormal{Acc}(h, v, v_{\ell})\) defined in \cref{defi:ACC} contains two terms: the ``Bando term'' \(\alpha (V(h)-v)\) and the ``follow-the-leader'' term \(\beta  \tfrac{v_{\ell}-v}{h^2}\). The Bando term allows the follower to match its velocity \(v\) to its optimal velocity \(V\), which is determined by its space headway \(h\), while the follow-the-leader term forces the follower to match its velocity \(v\) to its leader's velocity \(v_{\ell}\).
\end{remark}
\begin{remark}[an example for the optimal velocity function \(V\)]\label{rem:optimal_velocity} Assume that all vehicles have the same length, denoted by \(l \in \mathbb{R}_{>0}\). Assume also that the maximum value for the optimal velocity function \(V\) is the same as the maximal allowable velocity of all vehicles, \(v_{\max} \in \R_{>0}\). Let \(c , d_{s} \in \R_{>0}\) be given parameters. One example of an optimal velocity function is as follows:
\begin{equation}
    V(\cdot) = v_{\max} \tfrac{\tanh{(c\cdot-d_s)}+\tanh{(l + d_s)}}{1+\tanh{(l + d_s)}}.
\label{eq:optimal_velocity_function_example}
\end{equation}
If we take the follower's space headway \(h\) as the argument for the optimal velocity function \(V\), then \cref{eq:assumption_V} indicates that its optimal velocity \(V\) increases with respect to its space headway \(h\) and that the rate of change of \(V\)
with respect to \(h\) is bounded above by \(L>0\).

\begin{figure}[ht]
\centering
\begin{tikzpicture}[scale=0.75]
 \begin{axis}[
    xmin=0.5, xmax=10,
    ymin=0, ymax=38,
    xlabel={Space headway \(h\)},
    ylabel={Optimal velocity \(V(h)\)},
    grid=both,
    major grid style={dashed,gray!75},
    minor grid style={dotted,gray!75},
]
\addplot[domain=0:10,color=blue, samples=100,smooth] {30*(tanh(x-2.5)+tanh(7))/(1+tanh(7))};
\addplot[domain=10:0,color=red, samples=2,smooth,dotted, very thick]{30} node[]{\(v_{\max}\)};
\legend{\(V(h)\), \(v_{\max}\)};
\end{axis}
\end{tikzpicture}
\caption{One choice for the optimal velocity function \(V\) in \cref{eq:optimal_velocity_function_example} with parameters \(c=1\), \(l=4.5\), \(d_{\textnormal{s}}=2.5\), and \(v_{\max}=30\).}
\end{figure}
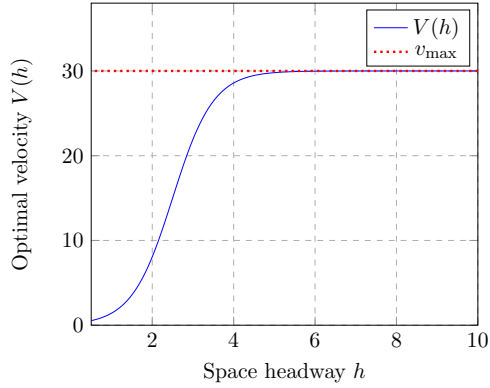 

\end{remark}

\begin{definition}
    The dynamics of the \(N+1\) vehicles are defined by the following coupled system of nonlinear ODEs:
\begin{equation}
\begin{aligned}
    \dot{x_i}(t)&=v_i(t), && i \in \{1, \dots, N+1\},~  t\in\R_{> 0}, \\
    v_1(t)&=v_{\ell}(t),  && t\in\R_{> 0}, \\
\dot{v}_i(t)&=\textnormal{Acc}\big(x_{i-1}(t)-x_i(t)-l,v_i(t),v_{i-1}(t)\big), &&  i \in \{2, \dots, N+1\},~  t \in\R_{> 0}, 
\end{aligned}
\label{system_dyn}
\end{equation}
where \((x_i(t), v_i(t))\) is the position--velocity vector of vehicle \(i \in \{1, \dots, N+1\}\) at time \(t\in \R_{>0}\). 
\end{definition}

\begin{definition}[a solution to the Bando-FtL system \eqref{system_dyn}]\label{defi:solution} Let \(T \in \R_{>0}\cup\{\infty\}\) be the considered time horizon, and suppose we are given the first vehicle's velocity \(v_{\ell} \colon \mathbb{R}_{\geq 0} \to (0, v_{\max}]\), with \( v_{\max} = 
\|V\|_{L^{\infty}(\mathbb{R}_{\ge 0})}
\). Let \(\big(\alpha, \beta, l \big)\in\R_{>0}^{3}\) be the parameters of \(\textnormal{Acc}\) as in \cref{defi:bando_acceleration}, and let the initial values \(({x}_{0},{v}_{0})\in\R^{N+1}\times\R^{N+1}_{\geq 0}\) be such that \[
x_{i-1, 0}-x_{i,0}-l>0,~  i \in \{2, \dots, N+1\}.\]
The mapping \(({x}, {v}) \in W^{2,\infty}((0,T); \R^{N+1}) \times W^{1,\infty}((0, T); \R_{\geq 0}^{N+1})\) is called a  solution of system \cref{system_dyn} if it satisfies the following integral equation:
\begin{align*}
    x_i(t)&=x_{i,0}+\int_0^t v_i(s) \dd s, && \forall  i \in \{1, \dots, N+1\}, \forall  t\in [0,T], \\
    v_i(t)&=v_{i,0}+\!\!\!\int_0^t \textnormal{Acc} \left(x_{i-1}(s)-x_i(s)-l,v_i(s),v_{i-1}(s)\right) \dd s, && \forall  i \in \{2, \dots, N+1\},~  \forall  t \in [0,T],\\
    v_1(t) &= v_{\ell}(t) && t\in[0,T].
\end{align*}

\end{definition}

The well-posedness for the Bando-FtL model over a finite time horizon \(T \in \R_{>0}\) has been established in \cite{gong2023well} recently, as follows.

\begin{theorem}\label{theo:existence_uniqueness_non_delayed}
Let \(T \in \R_{>0}\) be a finite time horizon. Let the assumptions in \cref{defi:solution} hold. Then there exists a unique solution \(({x},{v})\in W^{2,\infty}((0,T); \R^{N+1}) \times W^{1, \infty}((0,T); \R_{\geq 0}^{N+1})\) for system \cref{system_dyn} in the sense of \cref{defi:solution} over the finite time horizon \(T \in \R_{>0}\). In addition, there exists a minimum space headway \( d_{i, \min}(t)>0\) for each vehicle \(i \in \{2, \dots, N+1\}\) at any instant \(t \in [0, T]\). That is,
\begin{equation}
x_{i-1}(t)-x_i(t)-l\geq d_{i, \min}(t), \quad (t, i) \in [0, T] \times \{2, \dots, N+1\}.  \label{eq:d_min}
\end{equation} Furthermore, the solution \(({x},{v})\in W^{2,\infty}((0,T); \R^{N+1}) \times W^{1, \infty}((0,T); \R_{\geq 0}^{N+1})\) to system \cref{system_dyn} has a nonnegative bounded velocity (from above and below) and a bounded acceleration.
\end{theorem}

For the minimum space headway reported in \cref{theo:existence_uniqueness_non_delayed} (extracted from the recent work \cite{gong2023well}), \(d_{i, \min}(t) \to 0\) as \(t \to \infty\) for \(i \in \{2, \dots, N\}\), so this result cannot be used to analyze the long-term behavior of the Bando-FtL model. In the next section, we generalize it and demonstrate the well-posedness of the model over an infinite time horizon.

\section{The well-posedness and equilibrium of the Bando-FtL model on an infinite time horizon}
\label{sec: well_posedness}

In this section, we study the well-posedness of the Bando-FtL model, starting with the two-vehicle case in which the leader’s trajectory is given and the follower evolves according to \cref{defi:bando_acceleration}. We also analyze the equilibrium of the Bando-FtL model with \(N+1\) vehicles when the leader drives at a constant velocity.

\subsection{The well-posedness of the Bando-FtL model on an infinite time horizon for two vehicles}
We impose the following assumptions on the leader's trajectory and initial velocity. 
\begin{assumption}
Assume that the leader's velocity \(v_{\ell}(t) \in [v_{\min}, v_{\max}]\) is predetermined for \(t \in \R_{\geq 0}\), with \(v_{\max} = \|V\|_{L^{\infty}(\mathbb{R}_{\ge 0})} \geq v_{\min} > V(0)\). Denote by \(x_{\ell}(0) = x_{\ell, 0} \in \R\) the leader's initial position and by \(v_{\ell}(0) = v_{\ell ,0} \in [v_{\min}, v_{\max}]\) the leader's initial velocity. 
\label{assum_initial_leading}
\end{assumption}
Note that the leader's trajectory satisfies
\begin{equation}
\begin{aligned}
x_{\ell}(t) & =x_{\ell,0} + \int_0^t v_{\ell}(s) \dd s , & t \in \R_{\geq 0},
\end{aligned}
\label{eqn: leading}
\end{equation}
 where \(x_{\ell}(t)\) is the position of the leader at time \(t \in \R_{\geq 0}\). An assumption is also made for the follower's initial data.

\begin{assumption}
The follower has a nonnegative initial velocity and a strictly positive initial space headway. Specifically, the initial condition
\(
(x(0), v(0)) = (x_0, v_0) \in \mathbb{R} \times \mathbb{R}_{\ge 0}
\)
satisfies
\(
h(0) = h_0 \coloneqq x_{\ell,0} - x_0 - l > 0
\).

\label{assum_initial_following}
\end{assumption}

\begin{definition}
    The dynamics of the follower are governed by the following system of nonlinear ODEs:
\begin{equation}
\begin{aligned}
            \dot{x}(t) &= v(t), & t \in \R_{>0},\\
            \dot{v}(t) &= \alpha \left(V(h(t))-v(t)\right)+\beta \tfrac{v_{\ell}(t)-v(t)}{h^2(t)}, & t \in \R_{> 0}.
\end{aligned}\label{dyn_2_vehicle}
\end{equation}
 Here, \(x(t)\) and \(v(t)\) are the position and velocity of the follower at time \(t \in \R_{\geq 0}\), respectively, \(l \in \R_{>0}\) is the vehicle length, and the space headway of the follower is \(h \coloneqq x_{\ell} -x-l\). 
\end{definition}

The existence of a unique solution to \cref{eqn: leading,dyn_2_vehicle} over any finite time horizon, with the initial conditions given in \cref{defi:solution}, was established in \cite{gong2023well}. To prove the well-posedness of the Bando-FtL model on an infinite time horizon in the two-vehicle case, it suffices to show that the follower's space headway admits a uniform positive lower bound, since the right-hand side of the second equation in \cref{dyn_2_vehicle} is globally Lipschitz whenever the intervehicle distance is bounded away from zero. We therefore need the following result.

\begin{theorem}[existence of a uniform minimum distance over a small time interval \(T\)] Consider the dynamics of the leader and follower governed by  \cref{eqn: leading,dyn_2_vehicle} and suppose that \cref{assum_initial_leading,assum_initial_following} hold. Then the follower's space headway \(h\) is uniformly bounded from below on $[0,T]$ for some small time interval $T > 0$.
To be more precise, we have
\[ h(t) \geq h_{\min} > 0,    \forall t \in [0,T],\]
where
\begin{equation}
    h_{\min} \coloneqq \min\left\{\tfrac{A_0 + \sqrt{A_0^2 + 4 \alpha \beta}}{2 \alpha}, h_0, V^{-1}(v_{\min})\right\} >0,
    \label{eqn_h_min}
\end{equation}
with \(A_0 \coloneqq  -v_{\max} + \alpha h_0 - \tfrac{\beta}{h_0} \).
\label{theorem: h_min}
\end{theorem}

\begin{proof}
As has been shown in \cite{gong2023well} via a Picard--Lindel\"{o}f type argument, for each \(T^{*} \in \mathbb{R}_{>0}\), there is a unique classical solution to \cref{eqn: leading,dyn_2_vehicle} with the initial conditions as in \cref{defi:solution}. As $v_0 \leq v_{\max}$ thanks to our assumptions, it follows that the velocity of the follower has a uniform-in-time upper bound \( v(t) \leq v_{\max}\) for all $t\geq 0$ because of the stated dynamics in \cref{dyn_2_vehicle}.

Let us consider the difference in the change of velocities on the time horizon \([0, T^{*}]\). For any \(t \in [0, T^{*}]\),
\begin{align*}
\dot{v}_{\ell}(t) - \dot{v}(t) & = \dot{v}_{\ell}(t) - \alpha  \left(V(h(t) - v(t)\right) - \beta \tfrac{v_\ell(t) - v(t)}{h^2(t)}.
\end{align*}
From \cref{eq:assumption_V}, the optimal velocity function \(V\) is strictly monotone and thus invertible. In particular, by Taylor expanding the optimal velocity function \(V(h(t))\) at \(h_{\ell}(t) \coloneqq V^{-1}(v_{\ell}(t))\), we have that there exists \(\xi_{v_{\ell}}(t)\) lying in between \(h(t)\) and \(h_\ell(t)\) such that \[V(h(t)) = V(h_\ell(t)) + V'(\xi_{v_{\ell}}(t)) (h(t) - h_\ell(t)) = v_\ell(t) + V'(\xi_{v_{\ell}}(t)) (h(t) - h_\ell(t)).\]
Thus, for any \(t \in [0, T^{*}]\),
\begin{equation}\label{eq:key_equality}
        \begin{aligned}
         \ddot{h}(t) = \dot{v}_{\ell}(t) - \dot{v}(t) & = \dot{v}_{\ell}(t) - \alpha \left(v_\ell(t) + V'(\xi_{v_{\ell}}(t))(h(t) - h_\ell(t))  - v(t)\right) - \beta \tfrac{v_\ell(t) - v(t)}{h^2(t)}\\
        & = \dot{v}_{\ell}(t) - \alpha \left(v_\ell(t) - v(t)\right) - \alpha V'(\xi_{v_{\ell}}(t))(h(t) - h_\ell(t)) - \beta\tfrac{v_\ell(t) - v(t)}{h^2(t)}.
        \end{aligned}
        \end{equation}
We now consider three different cases separately:
\begin{enumerate}

\item If $v_0 < v_{\ell,0}$, then $\dot{h}(0) > 0$. By continuity, there exists some $t_1 > 0$ such that $h(t) \geq \min\{h_0,V^{-1}(v_{\min})\}$ whenever $t \in (0,t_1)$. Now, if $t_1 = \infty$, then there is nothing to prove, so we may assume without loss of generality that $t_1 < \infty$. In this case, we also have $h(t_1) < \min\{h_0,V^{-1}(v_{\min})\} \leq V^{-1}(v_{\min})$ and $v(t_1) \geq v_\ell(t_1)$. 
\label{item:1}

\item If $v_0 \geq v_{\ell,0}$ and $h_0 \geq V^{-1}(v_{\min})$, so that $\min\{h_0,V^{-1}(v_{\min})\} = V^{-1}(v_{\min})$, then by continuity, we can conclude that there exists some $t_2 > 0$ such that $h(t) \in [\tfrac{1}{2}V^{-1}(v_{\min}), h_0]$ whenever $t \in (0,t_2)$. Without loss of generality, we can further assume that $t_2 < \infty$ together with $h(t_2) = \tfrac{1}{2} V^{-1}(v_{\min})$ and $v(t_2) \geq v_\ell(t_2)$. 

\label{item:2}
\item In the scenario where $v_0 \geq v_{\ell,0}$ and $h_0 < V^{-1}(v_{\min})$, we will show the existence of a universal positive constant $\underbar{h} > 0$ depending solely on $\alpha$, $\beta$, $h_0$, and the specific choice of the optimal velocity function $V$ such that $h(t) \geq \underbar{h}$ whenever $t\in  [0,\tau)$, for some $\tau > 0$. In particular, since the uniform-in-time lower bound $\underbar{h}$ derived in this case does not depend on $v_0$, we can simply reuse the starting times from the two cases above (i.e., view $t_1$ or $t_2$ as the initial time $0$) to reduce them to this case. To proceed, we notice that the assumptions $v_0 \geq v_{\ell,0}$ and $h_0 < V^{-1}(v_{\min})$ guarantee the existence of some $\tau > 0$ such that $h(t) \leq V^{-1}(v_{\min})$ for all $t \in [0,\tau)$. Since \(V' > 0\) by assumption \cref{eq:assumption_V}, it follows that
\[
V'\!\bigl(\xi_{v_{\ell}}(t)\bigr)\bigl(h(t) - h_\ell(t)\bigr) \leq 0
\qquad \forall t \in [0, \tau).
\]
Therefore, 
\begin{align}
    \ddot{h}(t) \geq \dot{v}_{\ell}(t) - \alpha \left(v_\ell(t) - v(t)\right)  - \beta\tfrac{v_\ell(t) - v(t)}{h^2(t)}\qquad \forall t \in [0, \tau).
    \label{eq:key_equality_1}
\end{align}
 Consequently, upon integration, we can deduce from \eqref{eq:key_equality_1} that
\begin{equation}
\label{eq:key_inequality}
    \begin{aligned}
    \dot{h}(t) & = v_{\ell}(t) - v(t)    \\
    & \geq v_{\ell, 0} - v_0 + \int_0^t  \dot{v}_{\ell}(s) \dd s  - \alpha \int_0^t \left(v_\ell(s) - v(s)\right) \dd s  -\beta \int_0^t \tfrac{v_\ell(s) - v(s)}{h^2(s)} \dd s \\
    &= v_{\ell}(t) - v_0 - \alpha (h(t)-h_0) + \tfrac{\beta}{h(t)} - \tfrac{\beta}{h_0}\\
     & \geq -v_{\max} - \alpha (h(t)-h_0) + \tfrac{\beta}{h(t)} - \tfrac{\beta}{h_0}\\
     &= \tfrac{1}{h(t)}\left(-\alpha h^2(t) + \left(-v_{\max} + \alpha h_0 - \tfrac{\beta}{h_0}\right) h(t) + \beta\right)
    \end{aligned}
    \end{equation}
    for all $t \in [0,\tau)$. Therefore, it holds that $h(t) \geq \min\{h_0,\underbar{h}\}$ for all $t \in [0,\tau)$, in which 
    \[\underbar{h} \coloneqq \tfrac{1}{2 \alpha} \left(-v_{\max} + \alpha h_0 - \tfrac{\beta}{h_0} + \sqrt{\left(-v_{\max} + \alpha h_0 - \tfrac{\beta}{h_0}\right)^2 + 4 \alpha \beta}\right) > 0.\]
    Indeed, since the right-hand side of \eqref{eq:key_inequality} contains a quadratic polynomial in $h(t)$ with $-\alpha h^2(t) + \left(-v_{\max} + \alpha h_0 - \tfrac{\beta}{h_0}\right) h(t) + \beta \geq 0$ whenever $h(t) \in [0,\underbar{h}]$, we deduce that $\dot{h}(t) \geq 0$ whenever $h(t)\leq \underbar{h}$. Thus, $h(t) \geq \min\{h_0,\underbar{h}\}$ for all $t \in [0,\tau)$.
\label{item:3}
\end{enumerate}
Putting the results from these three cases together ensures the existence of some (a priori finite) time $T=\min\{t_1, t_2, \tau\} > 0$ such that $h(t) \geq \min\{V^{-1}(v_{\min}), h_0, \underbar{h}\}$ for all $t \in [0,T]$. In particular, notice that \(h_{\min} \coloneqq \min\left\{\tfrac{A_0 + \sqrt{A_0^2 + 4 \alpha \beta}}{2 \alpha}, h_0, V^{-1}(v_{\min})\right\} \) is independent of time.

\end{proof}
\begin{remark}[further discussion on \(h_{\min}\)]
We observe that
\[
\tfrac{A_0 + \sqrt{A_0^2 + 4 \alpha \beta}}{2 \alpha} \leq h_0,
\quad \text{where} \quad
A_0 \coloneqq -v_{\max} + \alpha h_0 - \tfrac{\beta}{h_0}.
\]
This directly implies that the uniform lower bound on the follower’s space headway, as given in \cref{eqn_h_min}, satisfies
\begin{equation}
\label{eqn_h_min_simplied}
h_{\min} 
= \min\!\left\{
\tfrac{A_0 + \sqrt{A_0^2 + 4 \alpha \beta}}{2 \alpha},
  V^{-1}(v_{\min})
\right\}
\leq h_0.
\end{equation}
This can be justified by the following straightforward computation:
\begin{align*}
\tfrac{A_0 + \sqrt{A_0^2 + 4 \alpha \beta}}{2 \alpha}
&= \tfrac{-v_{\max} + \alpha h_0 - \tfrac{\beta}{h_0}
+ \sqrt{\!\bigl(-v_{\max} + \alpha  h_0 - \tfrac{\beta}{h_0}\bigr)^2 + 4 \alpha \beta}}{2  \alpha}\\[1ex]
&= \tfrac{-v_{\max} + \alpha h_0 - \tfrac{\beta}{h_0}
+ \sqrt{v_{\max}^2 + \alpha^2 h_0^2 + \tfrac{\beta^2}{h_0^2}
- 2 \alpha v_{\max} h_0 + 2 \beta \tfrac{v_{\max}}{h_0}
+ 2 \alpha  \beta}}{2  \alpha}\\[1ex]
&\leq \tfrac{-v_{\max} + \alpha  h_0 - \tfrac{\beta}{h_0}
+ \sqrt{v_{\max}^2 + \alpha^2  h_0^2 + \tfrac{\beta^2}{h_0^2}
+ 2 \alpha v_{\max} h_0 + 2 \beta \tfrac{v_{\max}}{h_0}
+ 2 \alpha  \beta}}{2 \alpha}\\[1ex]
&= \tfrac{-v_{\max} + \alpha  h_0 - \tfrac{\beta}{h_0}
+ \sqrt{\!\bigl(v_{\max} + \alpha h_0 + \tfrac{\beta}{h_0}\bigr)^2}}{2 \alpha}
= h_0.
\end{align*}
Moreover, we note that this uniform minimum distance \(h_{\min}\) remains constant over time and does not depend on the velocities of either the follower or its leader. Instead, it is entirely determined by the model parameters \(\alpha\), \(\beta\), \(v_{\max}\), and \(v_{\min}\) and the initial space headway \(h_0\) of the follower. This property is crucial for extending the existence of a minimum distance from a small time interval to arbitrarily large intervals.
\label{rem_h_min}
\end{remark}

\begin{corollary}[existence of a uniform minimum distance over an infinite time horizon]
Suppose \cref{assum_initial_leading,assum_initial_following} hold. Then the follower's space headway \(h\) admits a uniform positive lower bound that persists over an infinite time horizon.
\label{coro:cor1}
\end{corollary}
\begin{proof}
Thanks to the content of \cref{theorem: h_min}, there exists some $T > 0$ for which $h(t) \geq h_{\min} > 0,    \forall t \in [0,T]$. Moreover, according to \cref{rem_h_min}, the aforementioned lower bound $h_{\min}$ on $h(t)$ does not depend on $T$ or the initial velocity. Thus, to complete the proof, it suffices to argue that $T = \infty$. Assume to the contrary that $T < \infty$. By continuity, there exists some $\delta > 0$ such that $h(T) = h_{\min}$ and $h(t) < h_{\min}$ for all $t\in (T, T+\delta)$. However, the proof of \cref{theorem: h_min} implies that $\dot{h}(t) \geq 0$ whenever $h(t) \leq \underbar{h}$. Consequently, at time \(t = T\), we must have \(\dot{h}(T) \ge 0\), which contradicts the assumption that \(h(t) < h_{\min} \le \underline{h}\) for all \(t \in (T,  T+\delta)\). As a result, $T = \infty$. 
\end{proof}

A numerical simulation illustrating the above result is shown in \cref{fig:lower_bounds}. 
In the two-vehicle dynamics \cref{dyn_2_vehicle}, we choose the parameters 
\(\alpha = 0.5\), \(\beta = 20\), \(c = 1\), \(l = 4.5\), \(v_{\max} = 30\), 
\(v_{\min} = 3\), and \(d_s = 2.5\). 
To examine the uniform lower bound on the follower’s space headway, we take the initial velocity 
of the follower to be \(v_{0} = v_{\max}\) and that of the leader to be \(v_{\ell, 0} = 0.35\, v_{\max}\), with an initial 
space headway \(h_0 = 10\). 
The leader has acceleration \(u_\ell(t) = -2\sin t, \forall t \in [0, T]\), so that it decelerates initially. 
The simulation is carried out over the time interval \([0,T]\) with \(T = 25\).

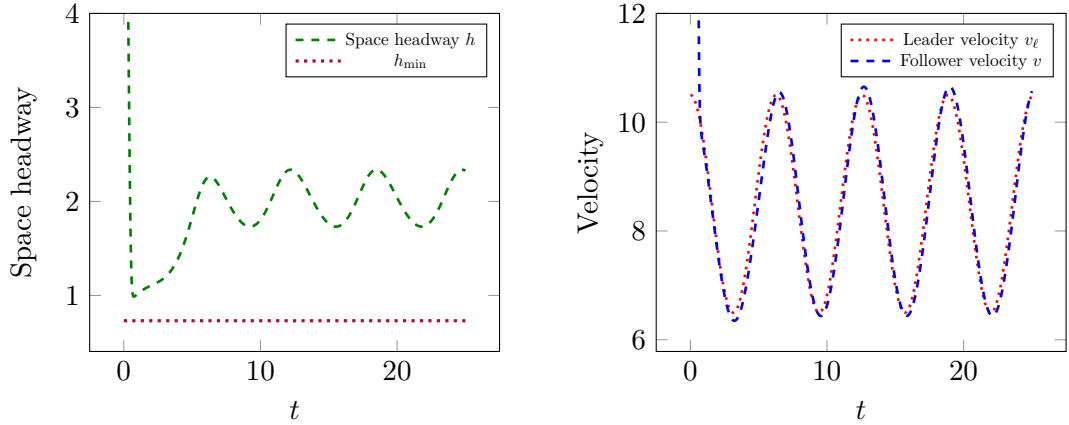
\begin{figure}
\centering
    \begin{tikzpicture}
         \begin{axis}[
             width=0.45\textwidth,
             xlabel={$t$},
             ylabel={Space headway}, 
            ylabel style={yshift=-10pt},
                 legend pos={north east,ymax=4,
                 legend style={nodes={scale=0.6, transform shape}}}
               ]
         \addplot[color=darkgreen, line width=1.0pt, mark=none, dashed]  table [y index=1, x index=0, col sep=comma] {Data/two_vehicle_headway.csv};
         \addplot[color=purple, samples=2,smooth,dotted, very thick]  table [y index=1, x index=0, col sep=comma] {Data/two_vehicle_headway_min.csv};  
           \addlegendentry{Space headway \(h\)};
         \addlegendentry{\(h_{\min}\)};
         \end{axis}
     \end{tikzpicture}
     \qquad
    \begin{tikzpicture}
        \begin{axis}[
            width=0.45\textwidth,
            xlabel={$t$},
            ylabel={Velocity},
            ylabel style={yshift=-10pt},
                legend pos={north east,ymax=12,
                 legend style={nodes={scale=0.6, transform shape}}}
                 ]
        \addplot[color=red,line width=1.0pt, mark=none, dotted]  table [y index=1, x index=0, col sep=comma] {Data/Leader_velocity.csv};
          \addplot[color=blue,line width=1.0pt, mark=none, dashed]  table [y index=1, x index=0, col sep=comma] {Data/Follower_velocity.csv};
         \addlegendentry{Leader velocity \(v_{\ell}\)};
         \addlegendentry{Follower velocity \(v\)};
        \end{axis}
    \end{tikzpicture}
 \caption{{\bf Left:} The follower's space headway, with the uniform minimum distance shown as a dotted line. {\bf Right:} The two vehicles' velocities.}
\label{fig:lower_bounds}
\end{figure}

Now we turn to the existence of a uniform maximum distance between the two vehicles. It can be seen that if the leader always drives at a constant velocity greater than the follower's velocity, then the follower's space headway is not bounded from above. However, we can expect an analogue of \cref{theorem: h_min} to hold under a similar set of assumptions.

\begin{theorem}[existence of a uniform maximum distance]
\label{theorem: h_max}
Consider the dynamics of the leader and follower governed by  \cref{eqn: leading,dyn_2_vehicle}, coupled with the initial datum \\ \((x_{\ell}(0), x(0), v(0),v_{\ell}(0) ) = (x_{\ell, 0}, x_0, v_0, v_{\ell,0}) \in \mathbb{R}^2 \times (0, v_{\max}] \times [v_{\min}, \bar{v}_{\max}]\) such that \(x_{\ell, 0} - x_0 -l >0\) and $\bar{v}_{\max} < v_{\max}$. Assume that the leader's velocity \(v_{\ell} \in L^{\infty}(\R_{\geq 0}; [v_{\min}, \bar{v}_{\max}]) \) and that the optimal velocity function \(V\) satisfies \cref{eq:assumption_V}. Then, the follower's space headway \(h\) is uniformly bounded from above at all times. In fact, we have
\[ h(t) \leq h_{\max},    \forall t \in \mathbb{R}_{\geq 0},\]
where
\begin{equation}
h_{\max} \coloneqq \max\left\{\tfrac{B_0 + \sqrt{B_0^2 + 4 \alpha \beta}}{2 \alpha}, h_0, V^{-1}(\bar{v}_{\max})\right\} < \infty
\label{eqn_h_max}
\end{equation}
with \(B_0 \coloneqq  \bar{v}_{\max} + \alpha h_0 - \tfrac{\beta}{h_0} \).
\end{theorem}

\begin{proof}
The proof closely mirrors the reasoning used in the proof of \cref{theorem: h_min}, but for the sake of completeness, we briefly sketch it. Adopting the same notation, we arrive at the identity in \cref{eq:key_equality}. We then discuss several scenarios separately:
\begin{enumerate}[label=(\roman*)]

\item If $v_0 > v_{\ell,0}$, then $\dot{h}(0) < 0$. Continuity implies that there exists some $t_3 > 0$ for which 
$h(t) \leq \max\{h_0,V^{-1}(\bar{v}_{\max})\}$ whenever $t \in (0,t_3)$. Without loss of generality, we may assume that $t_3 < \infty$ and $h(t_3) > \max\{h_0,V^{-1}(\bar{v}_{\max})\}$, as well as $v(t_3) \leq v_\ell(t_3)$.

\item If $v_0 \leq v_{\ell,0}$ and $h_0 \leq V^{-1}(\bar{v}_{\max})$, then by continuity we can find some $t_4 > 0$ such that $h(t) \in [h_0, \tfrac{3}{2}V^{-1}(\bar{v}_{\max})]$ whenever $t \in [0,t_4)$. Without loss of generality, we may assume that $t_4 < \infty$ and $h(t_4) = \tfrac{3}{2}V^{-1}(\bar{v}_{\max})$, as well as $v(t_4) \leq v_\ell(t_4)$.

\item In the scenario where $v_0 \leq v_{\ell,0}$ and $h_0 > V^{-1}(\bar{v}_{\max})$, we will show the existence of a universal positive constant $\bar{h} > 0$ depending only on $\alpha$, $\beta$, $h_0$, and the specific choice of the optimal velocity function $V$ such that $h(t) \leq \bar{h}$ whenever $t\in [0,\gamma)$, for some $\gamma > 0$. In particular, since the uniform-in-time upper bound $\bar{h}$ derived in this case does not depend on $v_0$, we can simply reuse the starting times in cases (i) and (ii) above (i.e., view $t_3$ or $t_4$ as the initial time $0$) to reduce the previous two scenarios (i) and (ii) to this case. To proceed, we notice that the assumptions $v_0 \leq v_{\ell,0}$ and $h_0 > V^{-1}(\bar{v}_{\max})$ guarantee the existence of some $\gamma > 0$ such that $h(t) \geq V^{-1}(\bar{v}_{\max})$ for all $t \in [0,\gamma)$. Since \(V' > 0\) due to the assumption \cref{eq:assumption_V}, and recalling that $h_\ell(t) = V^{-1}(v_\ell(t))$ and \(v_{\ell}(t) \leq \bar{v}_{\max}\), it follows that
\[
V'\!\bigl(\xi_{v_{\ell}}(t)\bigr)\bigl(h(t) - h_\ell(t)\bigr) \geq 0
\quad \forall t \in [0, \gamma).
\]
Therefore,
\begin{align}
    \ddot{h}(t) \leq \dot{v}_{\ell}(t) - \alpha \left(v_\ell(t) - v(t)\right)  - \beta\tfrac{v_\ell(t) - v(t)}{h^2(t)} \quad \forall t \in [0, \gamma).
    \label{eq:key_equality_2}
\end{align}
 Consequently, upon integration, we can deduce from \eqref{eq:key_equality_2} that
    \begin{equation}\label{eq:key_inequality_2}
    \begin{aligned}
    \dot{h}(t) & = v_{\ell}(t) - v(t) \\
    & \leq v_{\ell, 0} - v_0 + \int_0^t  \dot{v}_{\ell}(s) \dd s  - \alpha \int_0^t \left(v_\ell(s) - v(s)\right) \dd s  -\beta \int_0^t \tfrac{v_\ell(s) - v(s)}{h^2(s)} \dd s \\
    &= v_{\ell}(t) - v_0 - \alpha (h(t)-h_0) + \tfrac{\beta}{h(t)} - \tfrac{\beta}{h_0}\\
     &\leq \bar{v}_{\max} - \alpha (h(t)-h_0) + \tfrac{\beta}{h(t)} - \tfrac{\beta}{h_0}\\
     &= \tfrac{1}{h(t)}\left(-\alpha h^2(t) + \left(\bar{v}_{\max} + \alpha h_0 - \tfrac{\beta}{h_0}\right) h(t) + \beta\right)
    \end{aligned}
    \end{equation}
    for all $t \in [0,\gamma)$. Thus, it holds that $h(t) \leq \max\{h_0,\bar{h}\}$ for all $t \in [0,\gamma)$, in which 
    \[\bar{h} \coloneqq \tfrac{1}{2 \alpha} \left(\bar{v}_{\max} + \alpha h_0 - \tfrac{\beta}{h_0} + \sqrt{\left(\bar{v}_{\max} + \alpha h_0 - \tfrac{\beta}{h_0}\right)^2 + 4 \alpha \beta}\right) > 0.
    \]
    Indeed, since the right-hand side of \eqref{eq:key_inequality_2} contains a quadratic polynomial in $h(t)$ with $-\alpha h^2(t) + \left(-\bar{v}_{\max} + \alpha h_0 - \tfrac{\beta}{h_0}\right) h(t) + \beta \leq 0$ whenever $h(t) \geq \bar{h}$, we deduce that $\dot{h}(t) \leq 0$ whenever $h(t)\geq \bar{h}$. Consequently, $h(t) \leq \max\{h_0,\bar{h}\}$ for all $t \in [0,\gamma)$.
\end{enumerate}
Together, these three cases guarantee the existence of an a priori finite time \\ $\bar{T}=\min\{t_3, t_4, \gamma\} > 0$ such that $h(t) \leq \max\{V^{-1}(\bar{v}_{\max}), h_0, \bar{h}\}$ for all $t \in [0,\bar{T}]$. Finally, arguing as in the proof of \cref{coro:cor1} allows us to conclude that $\bar{T} = \infty$.
\end{proof}

\Cref{fig:upper_bounds} depicts simulations of the two-vehicle dynamics \cref{dyn_2_vehicle} with the same parameters as in \cref{fig:lower_bounds}, 
with \(\bar v_{\max} = 0.99\, v_{\max}\). 
To illustrate the uniform upper bound on the distance, we set the initial velocity of the follower 
to \(v_0 = 0.005\, v_{\max}\) and that of the leader to \(v_{\ell,0} = \bar v_{\max}\), and we assume an 
initial space headway of \(h_0 = 2\). 
The simulation is again carried out over the time interval \([0,T]\) with \(T = 20\).

\begin{figure}
\centering
    \begin{tikzpicture}
         \begin{axis}[
             width=0.45\textwidth,
             xlabel={ $t$ },
             ylabel={Space headway}, 
            ylabel style={yshift=-10pt},
                 legend pos={south east, ymin=30,
                 legend style={nodes={scale=0.6, transform shape}}}
               ]
         \addplot[color=darkgreen,line width=0.5pt, mark=none, dashed]  table [y index=1, x index=0, col sep=comma] {Data/two_vehicle_headway_2.csv};
         \addplot[color=purple, samples=2,smooth,dotted, very thick]  table [y index=1, x index=0, col sep=comma] {Data/two_vehicle_headway_max.csv};  
           \addlegendentry{Space headway \(h\)};
         \addlegendentry{\(h_{\max}\)};
         \end{axis}
     \end{tikzpicture}
     \qquad
    \begin{tikzpicture}
        \begin{axis}[
            width=0.45\textwidth,
            xlabel={$t$},
            ylabel={Velocity},
            ylabel style={yshift=-10pt},
                legend pos={south east, ymin=20,
                 legend style={nodes={scale=0.6, transform shape}}}
                 ]
        \addplot[color=red,line width=0.5pt, mark=none, dashed]  table [y index=1, x index=0, col sep=comma] {Data/Leader_velocity_2.csv};
          \addplot[color=blue,line width=0.5pt, mark=none, dashed]  table [y index=1, x index=0, col sep=comma] {Data/Follower_velocity_2.csv};
         \addlegendentry{Leader velocity \(v_{\ell}\)};
         \addlegendentry{Follower velocity \(v\)};
        \end{axis}
    \end{tikzpicture}
 \caption{{\bf Left:} The follower's space headway with the uniform maximum distance shown as a dotted line. {\bf Right:} The two vehicles' velocities.}
\label{fig:upper_bounds}
\end{figure}
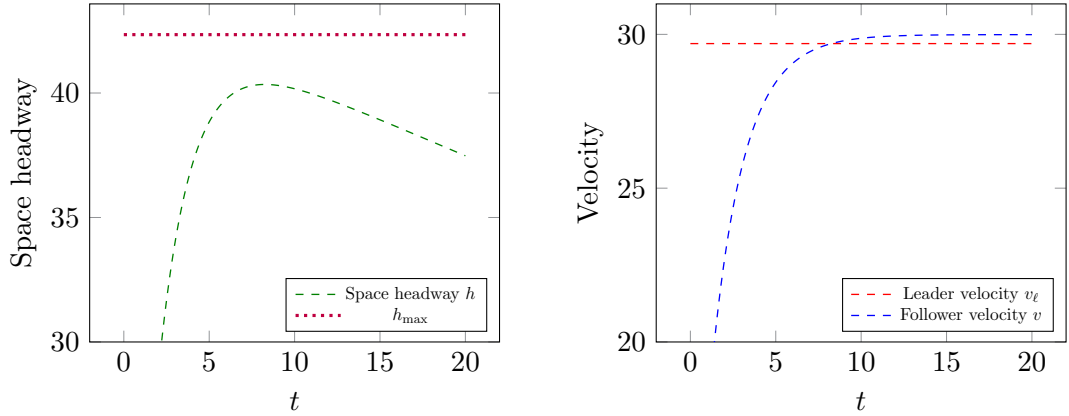

\begin{remark}[uniform bounds for velocities and accelerations]
\label{rem:uniform_bounds}
As an immediate corollary of the uniform bounds on the space headway, we deduce uniform bounds on the velocity and acceleration of the follower over an infinite time horizon. Specifically, under the conditions of \cref{theorem: h_min}, we have 
\[0\leq v(t) \leq v_{\max} \quad \text{and} \quad -\alpha v_{\max} - \beta \tfrac{v_{\max}}{h^2_{\min}} \leq a(t) \coloneqq \dot{v}(t) \leq \alpha v_{\max} + \beta \tfrac{v_{\max}}{h^2_{\min}}\ \forall t\in\R_{\geq0}.
\] 
If we also assume that $v(0) \geq V(h_{\min})$, then $v(t) \geq V(h_{\min})$ for all $t\geq 0$. Indeed, to justify the last statement on the uniform-in-time lower bound $\inf_{t \geq 0} v(t) \geq V(h_{\min})$ for the velocity of the follower, it suffices to notice that whenever $v(\tau) = V(h_{\min})$ for some $\tau \geq 0$, 
\begin{equation}
\label{eq:ref}
\begin{aligned}
\tfrac{\dd}{\dd t} v(t)\big\vert_{t=\tau} & = \alpha \left(V(h(\tau)) - v(\tau)\right) + \beta \tfrac{v_\ell(\tau) - v (\tau)}{h^2(\tau)} \\
& \geq \alpha\big(V(h(\tau)) - V(h_{\min})\big) + \beta \tfrac{v_{\min} - V(h_{\min})}{h^2(\tau)} \geq 0.
\end{aligned}
\end{equation}
\end{remark}

We now extend our result by establishing the well-posedness of the Bando-FtL model for 
\(N+1\) vehicles over an infinite time horizon, for any fixed $N\geq 1$.
\begin{assumption}
     The initial datum $\{(x_{i,0},v_{i,0})\}_{1\leq i\leq N+1}$ is chosen such that the initial space headways are $h_i(0) \coloneqq x_{i-1,0}-x_{i,0}-l = h_{i,0} > 0$ and $v_i(0) \in [v^i_0, v_{\max}]$ for all $2\leq i\leq N+1$, in which $v^i_0 \coloneqq V(h_{i,\min})$ and the (positive) sequence $\{h_{i,\min}\}_{2\leq i \leq N+1}$ is defined recursively via the following relations:
\begin{equation}\label{eq:h_seq}
h_{2,\min} \coloneqq h_{\min} \text{ and } h_{i+1,\min} \coloneqq \min\left\{\tfrac{A_{i,0} + \sqrt{A_{i,0}^2 + 4 \alpha \beta}}{2 \alpha},h_{i,\min}\right\}\quad \text{for $2\leq i\leq N$}, 
\end{equation} 
where 
\begin{equation}
 A_{i,0} = -v_{\max} + \alpha h_{i,0} - \tfrac{\beta}{h_{i,0}},\quad  2 \leq i \leq N.
\end{equation}
\label{assum_N_cars_initial}
\end{assumption}

\begin{corollary}[well-posedness of the Bando-FtL model with finitely many vehicles] 
\label{cor_well_posed_N_cars}
Assume that the dynamics of the first vehicle and the $N \geq 2$ followers are described by system \cref{system_dyn}  and let \cref{assum_initial_leading,assum_N_cars_initial} be satisfied.
Then the space headways of the followers admit a uniform positive lower bound. Specifically, for each $2\leq i\leq N+1$ and for all $t \geq 0$, it holds that $h_i(t) \geq h_{i,\min}$.
\end{corollary}

\begin{proof}
In the system \cref{system_dyn}, the dynamics of $\{(x_k(t),v_k(t))\}_{1 \leq k \leq i}$ evolve independently of $\{(x_m(t),v_m(t))\}_{i < m \leq N+1}$. Specifically, the state $(x_i,v_i)$ of the $i$th vehicle influences the next vehicle $(x_{i+1},v_{i+1})$, but not vice versa. According to \cref{theorem: h_min,rem:uniform_bounds}, we have $h_2(t) \geq h_{\min} > 0$ and $v_2(t) \geq V(h_{\min}) > V(0)$ for all $t\geq 0$. Thus, we may apply \cref{theorem: h_min} once more, this time treating the second vehicle as the leader of the third, which yields
\[h_3(t) \geq \min\left\{\tfrac{A_{2,0} + \sqrt{A_{2,0}^2 + 4 \alpha \beta}}{2 \alpha}, V^{-1}\left(V(h_{\min})\right)\right\} = \min\left\{\tfrac{A_{2,0}+ \sqrt{A_{2,0}^2 + 4 \alpha \beta}}{2 \alpha}, h_{\min}\right\}.\]
Since $v_3(0) \geq v^3_0 = V(h_{3,\min}) > V(0)$ by assumption, an argument analogous to that leading to \eqref{eq:ref} ensures that $v_3(t) \geq v^3_0$ for all $t \geq 0$. Then we can invoke \cref{theorem: h_min} again, this time treating the third vehicle as the leader of the fourth, leading us to 
\[h_4(t) \geq \min\left\{\tfrac{A_{3,0} + \sqrt{A_{3,0}^2 + 4 \alpha \beta}}{2 \alpha},V^{-1}\left(V(h_{3,\min})\right)\right\} = \min\left\{\tfrac{A_{3,0} + \sqrt{A_{3,0}^2 + 4 \alpha \beta}}{2 \alpha}, h_{3,\min}\right\}.\] The proof is then finished via an induction argument. 
\end{proof}
\begin{remark}[well-posedness for infinitely many vehicles]
\label{rem_well_posed_inifitely_cars}
The result of \cref{cor_well_posed_N_cars} can be extended to the case in which the number of vehicles $N$ tends to infinity under the additional assumption that the initial space headway \(h_{i,0} \geq \underbar{h}_0 >0\) for all \(i \geq 2\), where \(\underbar{h}_0\) is independent of \(N\). Indeed, denoting \(A_{0,0} \coloneqq -v_{\max} + \alpha \underbar{h}_0 - \tfrac{\beta}{\underbar{h}_0}\), we have $A_{0,0} \leq A_{i,0}$ and 
\begin{equation}
\label{ineqn_A_s}
\tfrac{A_{0,0} + \sqrt{A_{0,0}^2 + 4 \alpha \beta}}{2 \alpha} \leq \tfrac{A_{i,0} + \sqrt{A_{i,0}^2 + 4 \alpha \beta}}{2 \alpha}
\end{equation}
for all $i\geq 2$. Thus, by \cref{eqn_h_min_simplied,eq:h_seq,ineqn_A_s}, we have
\begin{equation}
    h_i(t) \geq \left\{\tfrac{A_{0,0} + \sqrt{A_{0,0}^2 + 4 \alpha \beta}}{2 \alpha }, V^{-1}(v_{\min})\right\}, \quad \forall~i \geq 2. 
\end{equation}
\end{remark}

This result is illustrated in \cref{fig_5_cars} for a system of five vehicles with initial velocities 
\(v_{1,0} = 10.5\), \(v_{2,0} = 16\), \(v_{3,0} = 22\), \(v_{4,0} = 26\), and \(v_{5,0} = 30\) 
and initial space headways 
\(h_{2,0} = 10\), \(h_{3,0} = 8\), \(h_{4,0} = 6\), and \(h_{5,0} = 5\).

\begin{figure}
\centering
    \begin{tikzpicture}
         \begin{axis}[
             width=0.45\textwidth,
             xlabel={ $t$ },
             ylabel={Space headway}, 
            ylabel style={yshift=-10pt},
                 legend pos={north east, ymax=4,
                 legend style={nodes={scale=0.6, transform shape}}}
               ]
         \addplot[color=blue,line width=0.5pt, mark=none, dashed]  table [y index=1, x index=0, col sep=comma] {Data/Follower_2_headway.csv};
         \addplot[color=yellow,line width=0.5pt, mark=none, dashed]  table [y index=1, x index=0, col sep=comma] {Data/Follower_3_headway.csv};
         \addplot[color=green,line width=0.45pt, mark=none, dashed]  table [y index=1, x index=0, col sep=comma] {Data/Follower_4_headway.csv};
         \addplot[color=purple,line width=0.5pt, mark=none, dashed]  table [y index=1, x index=0, col sep=comma] {Data/Follower_5_headway.csv};
         \addplot[color=red, samples=2,smooth,dotted, very thick]  table [y index=1, x index=0, col sep=comma] {Data/Five_vehicle_headway_min.csv};  
           \addlegendentry{Space headway \(h_2\)};
           \addlegendentry{Space headway \(h_3\)};
           \addlegendentry{Space headway \(h_4\)};
           \addlegendentry{Space headway \(h_5\)};
         \addlegendentry{\(h_{\min}\)};
         \end{axis}
     \end{tikzpicture}
     \qquad
    \begin{tikzpicture}
        \begin{axis}[
            width=0.45\textwidth,
            xlabel={$t$},
            ylabel={Velocity},
            ylabel style={yshift=-10pt},
                legend pos={north east, ymax=14,
                 legend style={nodes={scale=0.6, transform shape}}}
                 ]
        \addplot[color=red,line width=0.5pt, mark=none, dashed]  table [y index=1, x index=0, col sep=comma] {Data/Leader_1_velocity.csv};
        \addplot[color=blue,line width=0.5pt, mark=none, dashed]  table [y index=1, x index=0, col sep=comma] {Data/Follower_2_velocity.csv};
        \addplot[color=yellow,line width=0.5pt, mark=none, dashed]  table [y index=1, x index=0, col sep=comma] {Data/Follower_3_velocity.csv};
        \addplot[color=green,line width=0.5pt, mark=none, dashed]  table [y index=1, x index=0, col sep=comma] {Data/Follower_4_velocity.csv};
        \addplot[color=purple,line width=0.5pt, mark=none, dashed]  table [y index=1, x index=0, col sep=comma] {Data/Follower_5_velocity.csv};
    \addlegendentry{Leader velocity \(v_{1}\)};
         \addlegendentry{\(1\)st Follower velocity \(v_2\)};
         \addlegendentry{\(2\)nd Follower velocity \(v_3\)};
         \addlegendentry{\(3\)rd Follower velocity \(v_4\)};
         \addlegendentry{\(4\)th Follower velocity \(v_5\)};
        \end{axis}
    \end{tikzpicture}
 \caption{{\bf Left:} The four followers' space headways with the uniform minimum distance shown as a dotted line. {\bf Right:} The five vehicles' velocities.}
 \label{fig_5_cars}
\end{figure}
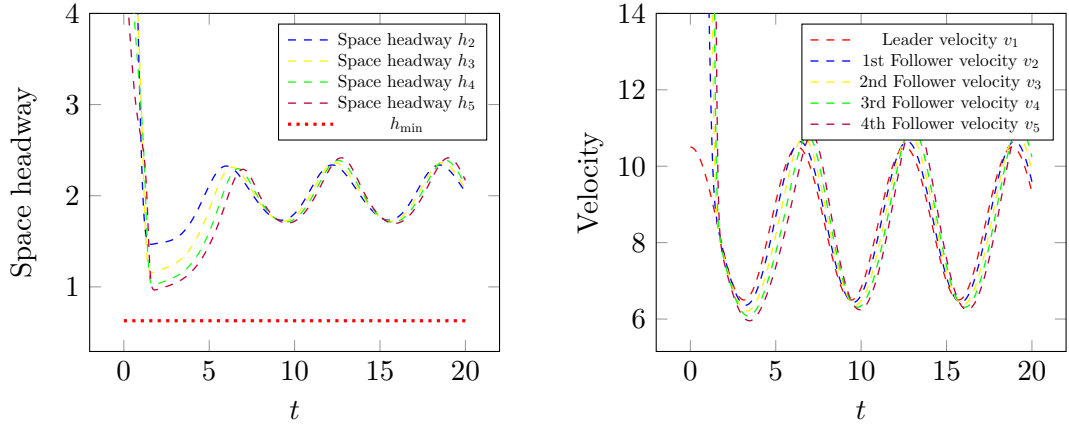

\subsection{The existence and uniqueness of equilibria of the Bando-FtL model}

Now we are ready to rewrite system \eqref{system_dyn} to study the dynamics of the \(N\) followers' space headways and relative velocities, under the following assumption.
\begin{assumption}
    The leader drives at a prescribed constant velocity \(v^{*} \in [v_{\min}, v_{\max}]\) with \(v_{\min}>V(0)\).
    \label{assum_leading_const_v}
\end{assumption}

Let \(h_i(t)=x_{i-1}(t) - x_i(t) -l\) and \(\Delta v_i (t)= v_{i-1}(t)-v_i(t)\) be the space headway and relative velocity of vehicle \(i \in \{2, \dots, N+1\}\) at time \(t \in \R_{\geq 0}\), respectively.  The Bando-FtL system \eqref{system_dyn} can be reformulated in terms of the space headways and relative velocities of vehicles \(i \in \{2,\dots,N+1\}\) as follows.

\begin{definition} The Bando-FtL system \eqref{system_dyn} admits the following equivalent representation: Fot \(t \in \R_{\geq 0}\), 
\begin{subequations}
\begin{align}
    v_1(t)&=v^{*}, \\
    \dot{h}_i(t) &=v_{i-1}(t) - v_i(t), \quad i \in \{2, \dots, N+1\}\label{eq:bando_position}, \\
    \Delta \dot{v}_2(t)&= -\textnormal{Acc}(x_1(t)-x_2(t)-l, v_2(t), v_1(t)), &&  \\
    \Delta \dot{v}_i(t)&=\textnormal{Acc}(x_{i-2}(t) - x_{i-1}(t)-l, v_{i-1}(t), v_{i-2}(t)) \\
    & \quad - \textnormal{Acc}(x_{i-1}(t) - x_i(t)-l, v_i(t), v_{i-1}(t)), \, i \in \{3, \dots, N+1\}, \label{eq:bando_velocity}\\
    (h_i(0), \Delta v_{i}(0))&=(x_{i-1, 0} - x_{i,0}-l, v_{i-1,0} - v_{i,0}), \quad i \in \{2, \dots, N+1\}.
    \label{eq:ego_initial_condition}
\end{align}
\label{system_dyn_new}
\end{subequations}

\end{definition}
\begin{remark}[well-posedness of the dynamics in relative position and velocity]
Under the conditions of \cref{{theo:existence_uniqueness_non_delayed}}, for any finite time horizon \(T \in \R_{>0}\), the mapping \((h, \Delta v) \in W^{2,\infty}((0,T); \R^{N}) \times W^{1,\infty}((0, T); \R^{N})\),
with

\begin{equation}
\begin{aligned}
    h_i & = x_{i-1} - x_{i} -l,  && \forall~ i \in \{2, \dots, N+1\}, \\
    \Delta v_i & = v_{i-1} - v_{i}, && \forall~ i \in \{2, \dots, N+1\},
\end{aligned}
\label{def_h_rv}
\end{equation}
is the unique solution to system \eqref{system_dyn_new} with the initial datum \cref{eq:ego_initial_condition}.
\end{remark}
At equilibrium, it is natural and desirable to have all vehicles driving at the same velocity as the first vehicle with equivalent space headway. This motivates the following observation.
\begin{lemma}[existence and uniqueness of equilibrium]\label{def: equilibrium} Let \(T \in \R_{>0}\cup\{\infty\}\) be given, and let \cref{assum_leading_const_v} be satisfied. Then there exists a unique equilibrium associated with the coupled system of nonlinear ODEs \cref{eq:bando_position,eq:bando_velocity}, which occurs when
$v_{i} = v^{*}$ and $h_i = h^{*}\coloneqq V^{-1}(v^{*})$ for all $i \in \{2, \dots, N+1\}$.
\end{lemma}
\begin{proof}
The strict monotonicity of the optimal velocity function \(V\) guarantees the existence and uniqueness of such an \(h^{*}\). Setting the right-hand side of  \cref{eq:bando_position} to zero, we have that \(v_{2} = v_{1} =v^{*}\), thanks to \(v_{1} \equiv v^{*}\), and inductively that \(v_{i} =v^{*}\ \forall i\in \{1, \dots, N+1\}\). Next, we insert \(v_i = v^{*}~   \forall  i \in\{1, \dots, N+1\}\) into \cref{eq:bando_velocity} to determine the corresponding space headway (at equilibrium). For $i=2$, we have $\textnormal{Acc}(x_{1}-x_2-l,v^{*},v^{*})=0$; thus, $V(x_{1}-x_{2}-l)=v^{*}$ or equivalently $h_2 = h^{*} = V^{-1}(v^{*})$. A straightforward induction argument allows us to deduce that $h_i = h^*$ for all $i \in \{3,\ldots, N+1\}$ as well.
\end{proof}

\begin{remark}
The result of \cref{def: equilibrium} remains valid in the case of infinitely many vehicles (i.e., \(N = \infty\)).
\end{remark}

\section{Long-term behavior of the Bando-FtL model}
\label{sec:sec_two_vehicle}

In this section, we first consider the two-vehicle case, in which the first vehicle drives at a constant velocity \(v^{*} \in [v_{\min}, v_{\max}]\) with \(v_{\min}>V(0)\) starting from the initial time \(t=0\), and the follower is described by the Bando-FtL model given in \cref{defi:bando_acceleration}. The long-term behavior of the Bando-FtL model in the two-vehicle case is illustrated in \cref{fig_2_cars_star}.
We investigate the long-term behavior of the follower by studying suitable energy or Lyapunov functions parameterized or weighted with various parameter assumptions. At the end of this section, we extend our result to the case of finitely many vehicles, with the first vehicle moving at a constant velocity. 

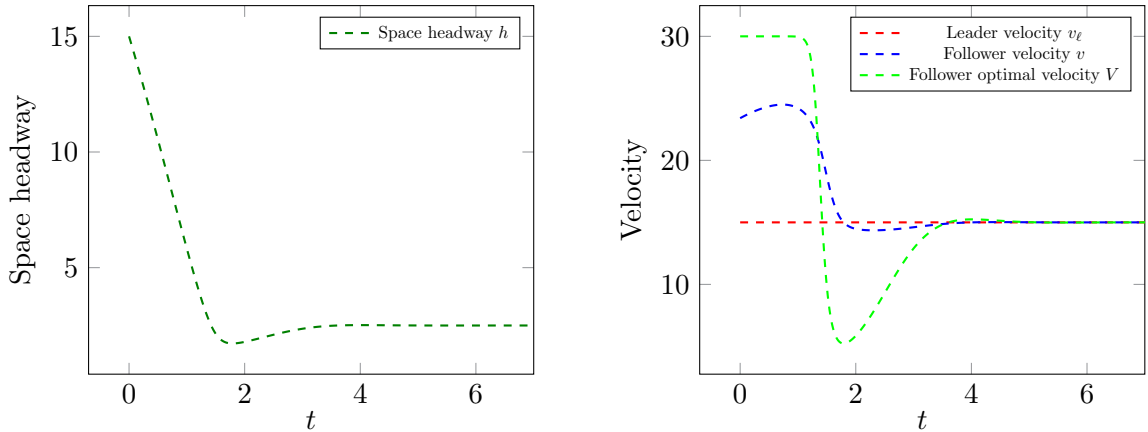
\begin{figure}[!ht]
\centering

\begin{minipage}{0.48\textwidth}
\centering
\begin{tikzpicture}
\begin{axis}[
    width=\textwidth,
    xlabel={$t$},
    ylabel={Space headway},
    legend pos=north east, xmax=7,
    legend style={nodes={scale=0.6, transform shape}},
    xlabel style={yshift=5pt},
    ylabel style={yshift=-10pt},
]
\addplot[color=darkgreen, line width=0.8pt, mark=none, dashed]
table [x index=0, y index=1, col sep=comma]
{Data/two_vehicle_headway_2_star.csv};
\addlegendentry{Space headway $h$}
\end{axis}
\end{tikzpicture}
\end{minipage}
\hfill
\begin{minipage}{0.48\textwidth}
\centering
\begin{tikzpicture}
\begin{axis}[
    width=\textwidth,
    xlabel={$t$},
    ylabel={Velocity},
    legend pos=north east, xmax=7,
    legend style={nodes={scale=0.6, transform shape}},
    xlabel style={yshift=5pt},
    ylabel style={yshift=-10pt},
]
\addplot[color=red,  line width=0.8pt, mark=none, dashed]
table [x index=0, y index=1, col sep=comma]
{Data/Leader_velocity_2_star.csv};
\addplot[color=blue, line width=0.8pt, mark=none, dashed]
table [x index=0, y index=1, col sep=comma]
{Data/Follower_velocity_2_star.csv};
\addplot[color=green, line width=0.8pt, mark=none, dashed]
table [x index=0, y index=1, col sep=comma]
{Data/Follower_opt_velocity_2_star.csv};

\addlegendentry{Leader velocity $v_\ell$}
\addlegendentry{Follower velocity $v$}
\addlegendentry{Follower optimal velocity $V$}
\end{axis}
\end{tikzpicture}
\end{minipage}

\caption{\textbf{Left}: Follower's space headway. \textbf{Right}: Leader and follower velocities with follower optimal velocity.}
\label{fig_2_cars_star}
\end{figure}

The leader's dynamics can be characterized as follows:
\begin{equation}
v_{\ell}(t) = v^{*} ~~\text{and}~~  x_{\ell}(t)=x_{\ell,0} +v^{*} t, \quad \forall t \in \R_{\geq 0}.
\label{eqn: leading_constant_velocity}
\end{equation}
The dynamics of the follower are dictated by the nonlinear ODE system \cref{dyn_2_vehicle} with an initial condition \(\left(x(0),v(0)\right) = (x_0,v_0)\in \R\times\R_{\geq 0}\) chosen such that the follower's initial space headway is \(h(0)=x_{\ell, 0} - x_0 -l > 0\).

\subsection{Asymptotic convergence in the two-vehicle Bando-FtL Model}

We first analyze the local behavior of system \eqref{system_dyn_new} around its equilibrium using linearization \cite{chou2024stability, hayat2023dissipation}. For the two-vehicle case, the system reduces to
\begin{equation}
    \begin{aligned}
        \dot{h}_2(t) & = \Delta v_2(t),\\
        \Delta \dot{v}_2(t) &= -\alpha \bigl(V(h_2(t)) + \Delta v_2(t) - v^{*}\bigr)
                              - \beta \tfrac{\Delta v_2(t)}{h_2^2(t)},
    \end{aligned}
    \label{eqn_new_2}
\end{equation}
where \((\alpha, \beta) \in \mathbb{R}^2_{>0}\) and \(v^{*} \in [v_{\min}, v_{\max}]\). 
By \cref{def: equilibrium}, \((h^{*}, 0) = (V^{-1}(v^{*}), 0)\) is the unique equilibrium of
\cref{eqn_new_2}. The Jacobian matrix \(J\) evaluated at \((h^{*}, 0)\) is
\[
J\big|_{(h^{*},0)} =
\begin{bmatrix}
    0 & 1\\[2pt]
    -\alpha V'(h^{*}) & -\alpha - \tfrac{\beta}{(h^{*})^2}
\end{bmatrix}.
\]
The eigenvalues of $J$ are
\[
\lambda_{1,2}
= \tfrac{
    -\Bigl(\alpha + \tfrac{\beta}{(h^{*})^2}\Bigr)
    \pm 
    \sqrt{
        \Bigl(\alpha + \tfrac{\beta}{(h^{*})^2}\Bigr)^2
        - 4 \alpha V'(h^{*})
    }
}{2}.
\]
Since $\alpha,\beta>0$, the equilibrium $(h^{*},0)$ is locally asymptotically stable if and only if $V'(h^{*})>0$. 

\begin{remark}
   Since this is a local result, it implies asymptotic convergence to the equilibrium \((h^{*}, 0)\) only for initial conditions sufficiently close to equilibrium.
\end{remark}

To investigate globally the long-term behavior of the solution of the ODE system \eqref{system_dyn} with two vehicles, we thus impose the following assumption.
\begin{assumption}
\label{assump_1}
The coefficient \(\beta\) in the Bando-FtL acceleration \cref{defi:ACC} satisfies
\begin{equation}\label{assump:beta}
\beta \geq \max\limits_{h_{\min} \leq h \leq h_{\max}} \big(V'(h) h^2\big), 
\end{equation}
where \(h_{\min}\) and \(h_{\max}\) are as defined in \cref{eqn_h_min,eqn_h_max}, respectively. 
\end{assumption}

We analyze the validity of this assumption by considering the optimal velocity function \(V\) defined in \cref{eq:optimal_velocity_function_example}.

\begin{remark}
    Informally speaking, \cref{assump_1} requires the constant \(\beta\) (measuring the strength of the ``follow-the-leader'' effect) to be moderately large. Define the map \(F\colon \R_{>0} \to \R_{\geq 0}\) with \(F(h) \coloneqq V'(h) h^2\). Considering the optimal velocity function \(V\) given in \cref{eq:optimal_velocity_function_example}, we have for each \(\forall h>0\) that
\begin{align*}
    F(h) = \tfrac{c  v_{\max}}{1+\tanh{(l + d_s)}}\left(1-\tanh^2(ch-d_s)\right) h^2.
\end{align*}
For this particular choice of $V$, the right-hand side of the bound \eqref{assump:beta} can be evaluated explicitly. Since
\begin{align*}
    F'(h) = 2\tfrac{c v_{\max}}{1+\tanh{(l + d_s)}} h \sech^3(d_s-ch) \left(c h \sinh(d_s-ch)+\cosh(d_s-ch)\right),
\end{align*}
the critical points of the function \(F\) located within $\R_{>0}$ must satisfy \(c \hat{h} \tanh(c \hat{h}-d_s) =1\). Moreover, as
\begin{align*}
F''(\hat{h}) & = -\tfrac{2 c v_{\max}}{\tanh(l+d_s)+1} \sech^2(d_s-c\hat{h}) (c \hat{h})^2<0,
\end{align*}
we deduce that \(F\) attains its maximum at \(h=\hat{h}\). Consequently,  \[F(h) \leq  F(\hat{h}) = \tfrac{c v_{\max}}{1+\tanh{(l + d_s)}}\left((\hat{h})^2-(\hat{h} \tanh(c \hat{h}-d_s))^2\right) = \tfrac{c v_{\max}}{1+\tanh{(l + d_s)}}\left((\hat{h})^2-\tfrac{1}{c^2}\right).\]
If we choose a standard set of model parameters \(v_{\max}=10\), \(l =4.5\), \(c=2\), and  \(d_s = 2.5\), then \(\hat{h} \approx 1.432\) and \(F(\hat{h}) \approx 18.01\). Thus, it suffices to impose that \(\beta \geq 18.1\).
\end{remark}

Now we are ready to prove some preliminary lemmas given \cref{assump:beta}. Consider the following six subregions of the region \(\{(v,V) \mid v \in \R_{\geq 0}, V \in \R_{\geq 0}\}\): 
\begin{align}
A \coloneqq & \{(v, V) \in \R_{\geq 0}^2 \mid v^* > v > V\}, ~~~~ B \coloneqq \{(v, V)\in \R_{\geq 0}^2 \mid v^{*} > V > v\},  \label{eqn: A_B} \\
C \coloneqq & \{(v, V) \in \R_{\geq 0}^2 \mid V > v^* > v\},~~~~ D \coloneqq \{(v, V)\in \R_{\geq 0}^2 \mid V > v > v^*\}, \label{eqn: C_D} \\  
E \coloneqq & \{(v, V)\in \R_{\geq 0}^2 \mid v > V >v^*\},~~~~ F \coloneqq \{(v, V)\in \R_{\geq 0}^2 \mid v > v^{*} > V\}. \label{eqn: E_F}
\end{align}
Refer to \cref{fig:phase_regions_and_decomposition} for an illustration of these six subregions. \Cref{fig:regions_trajectory} depicts a trajectory of the pair \((V,v)\) traversing some of the subregions, governed by \cref{dyn_2_vehicle,eqn: leading_constant_velocity}, starting from the green point in region \(D\) and converging to the red point \((v^{*}, v^{*})\). For any admissible initial datum \((x_0, v_0, x_{\ell, 0}) \in \R \times \R_{\geq 0} \times \R\) associated with \cref{eqn: leading,dyn_2_vehicle}, \((v_0, V(h_0))\) lies in one of the above six subregions or on their boundaries with \(h_0 = x_{\ell, 0} - x_0 - \ell\). 

\begin{figure}[!ht]
\centering

\newcommand{\plotsize}{6.2cm}

\def\arrowlen{0.65}

\def\XLabShiftX{-4pt}
\def\XLabShiftY{0pt}
\def\YLabShiftX{0pt}
\def\YLabShiftY{-4pt}
\def\xmaxR{6.2}
\def\ymaxR{6.2}
\def\vstarR{2.657}
\def\xmaxL{35}
\def\ymaxL{35}
\def\vstarL{15.0}

\begin{minipage}{0.48\textwidth}
\centering

\pgfmathsetmacro{\rightscale}{\plotsize/\xmaxR} 

\begin{tikzpicture}[>=Latex][
  width=\plotsize,
  height=\plotsize,
  scale only axis,
  xmin=0, xmax=\xmaxL,
  ymin=0, ymax=\ymaxL,
  axis lines=left,
  axis line style={->, line width=1.1pt},
  xtick=\empty, ytick=\empty,
  clip=false, 
  axis on top, 
xlabel style={at={(axis description cs:1,0.08)}, anchor=west},
ylabel style={
  rotate=270,
  at={(axis description cs:0.2,1)},
  anchor=south, xshift=0, yshift=0pt},
]

\fill[gray!20] (0,0) -- (0,\vstarR) -- (\vstarR,\vstarR) -- cycle;
\fill[gray!35] (0,0) -- (\vstarR,0) -- (\vstarR,\vstarR) -- cycle;
\fill[gray!45] (\vstarR,0) -- (\xmaxR,0) -- (\xmaxR,\vstarR) -- (\vstarR,\vstarR) -- cycle;
\fill[gray!60] (\vstarR,\vstarR) -- (\xmaxR,\vstarR) -- (\xmaxR,\ymaxR) -- cycle;
\fill[gray!30] (0,\vstarR) -- (0,\ymaxR) -- (\xmaxR,\ymaxR) -- (\vstarR,\vstarR) -- cycle;
\fill[gray!50] (0,\vstarR) -- (\vstarR,\vstarR) -- (\vstarR,\ymaxR) -- (0,\ymaxR) -- cycle;

\draw[->, line width=1.1pt] (0,0) -- (\xmaxR,0)
  node[anchor=west] {$V$};

\draw[->, line width=1.1pt] (0,0) -- (0,\ymaxR)
  node[anchor=south] {$v$};

\node[below] at (\vstarR,0) {$v^{*}$};
\node[left]  at (0,\vstarR) {$v^{*}$};

\draw[red, densely dotted, line width=1.2pt] (0,\vstarR) -- (\xmaxR,\vstarR);
\draw[red, densely dotted, line width=1.2pt] (\vstarR,0) -- (\vstarR,\ymaxR);
\draw[red, densely dotted, line width=1.2pt] (0,0) -- (\xmaxR,\ymaxR);

\node at (0.65*\vstarR,0.35*\vstarR) {\Large $B$};
\node at ({\vstarR + 0.55*(\xmaxR-\vstarR)},{0.35*\vstarR}) {\Large$C$};
\node at ({0.45*\vstarR},{\vstarR + 0.70*(\ymaxR-\vstarR)}) {\Large $F$};
\node at ({\vstarR + 0.55*(\xmaxR-\vstarR)},{\vstarR + 0.80*(\ymaxR-\vstarR)}) {\Large $E$};
\node at ({0.35*\vstarR},{0.70*\vstarR}) {\Large $A$};
\node at ({\vstarR + 0.70*(\xmaxR-\vstarR)},{\vstarR + 0.55*(\ymaxR-\vstarR)}) {\Large $D$};

\draw[->, blue] (0.5*\vstarR,\vstarR) -- ($(0.5*\vstarR,\vstarR)!\arrowlen!(0.5*\vstarR,\vstarR-0.5)$);
\draw[->, blue] (1.55,1.55) -- ($(1.55,1.55)!\arrowlen!(1.65,2.05)$);

\draw[->, blue] (4.65,4.65) -- ($(4.65,4.65)!\arrowlen!(4.55,4.05)$);
\draw[->, blue] (3.675,4.55) -- ($(3.675,4.55)!\arrowlen!(3.475,4.15)$);

\draw[->, blue] (2.225,1.55) -- ($(2.425,1.55)!\arrowlen!(2.429,1.908)$);
\draw[->, blue] (\vstarR,0.5*\vstarR) -- ($(\vstarR,0.5*\vstarR)!\arrowlen!(\vstarR+0.5,0.5*\vstarR)$);

\draw[->, blue] (1.5 * \vstarR,0.7 * \vstarR) -- ($(1.5 * \vstarR,0.7 * \vstarR)!\arrowlen!(1.5 * \vstarR + 0.3,0.7 * \vstarR + 0.4)$);
\draw[->, blue] (4.65,\vstarR) -- ($(4.65,\vstarR)!\arrowlen!(4.65,\vstarR+0.5)$);
\draw[->, blue] (\vstarR,4.65) -- ($(\vstarR,4.65)!\arrowlen!(\vstarR-0.5,4.65)$);
\draw[->, blue] (0.51*\vstarR,4.40) -- ($(0.51*\vstarR,4.40)!\arrowlen!(0.51*\vstarR-0.3,4.05)$);

\fill[red] (\vstarR,\vstarR) circle (1.1pt);

\end{tikzpicture}

\caption{Six subregions of the \((V,v)\) plane.}
\label{fig:phase_regions_and_decomposition}
\end{minipage}
\hfill
\begin{minipage}{0.48\textwidth}
\centering

\def\xmaxL{35}
\def\ymaxL{35}
\def\vstarL{15.0}

\begin{tikzpicture}[>=Latex]
\begin{axis}[
  width=\plotsize,
  height=\plotsize,
  scale only axis,
  xmin=0, xmax=\xmaxL,
  ymin=0, ymax=\ymaxL,
  axis lines=left,
  axis line style={->, line width=1.1pt},
  xtick=\empty, ytick=\empty,
  clip=false, 
  axis on top, 
xlabel={$V$},
ylabel={$v$},
xlabel style={at={(axis description cs:1,0.08)}, anchor=west},
ylabel style={
  rotate=270,
  at={(axis description cs:0.2,1)},
  anchor=south, xshift=0, yshift=0pt},
]

 \fill[gray!20] (axis cs:0,0) -- (axis cs:0,\vstarL) -- (axis cs:\vstarL,\vstarL) -- cycle;
 \fill[gray!35] (axis cs:0,0) -- (axis cs:\vstarL,0) -- (axis cs:\vstarL,\vstarL) -- cycle;
 \fill[gray!45] (axis cs:\vstarL,0) -- (axis cs:\xmaxL,0) -- (axis cs:\xmaxL,\vstarL) -- (axis cs:\vstarL,\vstarL) -- cycle;
 \fill[gray!60] (axis cs:\vstarL,\vstarL) -- (axis cs:\xmaxL,\vstarL) -- (axis cs:\xmaxL,\ymaxL) -- cycle;
 \fill[gray!30] (axis cs:0,\vstarL) -- (axis cs:0,\ymaxL) -- (axis cs:\xmaxL,\ymaxL) -- (axis cs:\vstarL,\vstarL) -- cycle;
 \fill[gray!50] (axis cs:0,\vstarL) -- (axis cs:\vstarL,\vstarL) -- (axis cs:\vstarL,\ymaxL) -- (axis cs:0,\ymaxL) -- cycle;


\node[below] at (axis cs:\vstarL,0) {$v^{*}$};
\node[left] at (axis cs:0,\vstarL) {$v^{*}$};

\draw[red, densely dotted, line width=1.2pt] (axis cs:0,\vstarL) -- (axis cs:\xmaxL,\vstarL);
\draw[red, densely dotted, line width=1.2pt] (axis cs:\vstarL,0) -- (axis cs:\vstarL,\ymaxL);
\draw[red, densely dotted, line width=1.2pt] (axis cs:0,0) -- (axis cs:\xmaxL,\ymaxL);

\node at (axis cs:{0.65*\vstarL},{0.35*\vstarL}) {\Large $B$};
\node at (axis cs:{\vstarL + 0.55*(\xmaxL-\vstarL)},{0.35*\vstarL}) {\Large $C$};
\node at (axis cs:{0.45*\vstarL},{\vstarL + 0.70*(\ymaxL-\vstarL)}) {\Large $F$};
\node at (axis cs:{\vstarL + 0.55*(\xmaxL-\vstarL)},{\vstarL + 0.80*(\ymaxL-\vstarL)}) {\Large $E$};
\node at (axis cs:{0.35*\vstarL},{0.70*\vstarL}) {\Large $A$};
\node at (axis cs:{\vstarL + 0.70*(\xmaxL-\vstarL)},{\vstarL + 0.55*(\ymaxL-\vstarL)}) {\Large $D$};

\addplot[blue, thick] table [x=V, y=v, col sep=comma] {Data/PhaseTrajectory_V_v.csv};

\addplot[
  only marks, mark=*, mark size=2.8pt, mark options={draw=green, fill=green}] table [x=V, y=v, col sep=comma] {Data/InitPoint_V_v.csv};

\addplot[only marks, mark=*, mark size=3pt, red] coordinates {(\vstarL,\vstarL)};

\end{axis}
\end{tikzpicture}

\caption{Trajectory and regions in the $(V,v)$ plane.}
\label{fig:regions_trajectory}
\label{fig:decomposition_of_regions}
\end{minipage}

\end{figure}

Suppose the initial datum \((x_0, v_0, x_{\ell,0})\in \R\times\R_{\ge0}\times\R\) associated with \cref{eqn: leading,dyn_2_vehicle} satisfies either
\[
(v_0,V(h_0))\in 
B \coloneqq \{(v,V)\in\R_{\ge0}^2 \mid v^{*}>V>v\}
\]
or
\[
(v_0,V(h_0))\in 
E \coloneqq \{(v,V)\in\R_{\ge0}^2 \mid v>V>v^{*}\}.
\]
Then one of the following two scenarios occurs:
\begin{enumerate}
    \item The pair \((v(t),V(h(t)))\) remains in its initial region for all time,
    \[
    (v(t),V(h(t)))\in B,\ \forall t\ge0
    \quad\text{or}\quad
    (v(t),V(h(t)))\in E,\ \forall t\ge0.
    \]
    \label{case:BE_1}
\item The pair \((v(t),V(h(t)))\) reaches the boundary of its initial region in finite time, that is, there exists a \(\hat{t}>0\) such that
    \[
    (v(\hat{t}),V(h(\hat{t})))\in \partial B
    \quad\text{or}\quad
    (v(\hat{t}),V(h(\hat{t})))\in \partial E,
    \]
    where
    \[
    \partial B
    =
    \{(v,V)\in\R_{\ge0}^2 \mid V=v^{*}\ge v\}
    \cup
    \{(v,V)\in\R_{\ge0}^2 \mid v=V<v^{*}\}
    \]
    and
    \[
    \partial E 
    =
    \{(v,V)\in\R_{\ge0}^2 \mid V=v^{*}\le v\}
    \cup
    \{(v,V)\in\R_{\ge0}^2 \mid v=V>v^{*}\}.
    \]
    \label{case:BE_2}
\end{enumerate}

If the pair \((v,V)\) stays within its initial region \(B\) or \(E\) for any finite time \(t\in \R_{\geq 0}\), as in the first scenario, then we have the following.
\begin{lemma}[{exponential convergence to equilibrium in regions \(B\) and \(E\)}]
Let the optimal velocity function \(V\) satisfy \cref{eq:assumption_V}, let the vehicle length be \(l>0\), and let \cref{assum_leading_const_v} hold. The follower's dynamics are governed by \cref{dyn_2_vehicle}.
 Suppose the initial condition \((x_0,x_{\ell,0},v_0)\in\R^2\times\R_{\ge0}\) satisfies either
\begin{equation}
\label{condition_B}
(v_0, V(h_0))\in B \coloneqq \{(v,V)\in\R_{\ge0}^2 \mid v^{*}>V>v\}
\quad\text{and}\quad
(v(t),V(h(t)))\in B,\ \forall t\ge0
\end{equation}
or
\begin{equation}
\label{condition_E}
(v_0, V(h_0))\in E \coloneqq \{(v,V)\in\R_{\ge0}^2 \mid v>V>v^{*}\}
\quad\text{and}\quad
(v(t),V(h(t)))\in E,\ \forall t\ge0.
\end{equation}
Then 
\begin{align}
    \lim \limits_{t\to \infty} v(t) = \lim \limits_{t \to \infty} V(h(t)) = v^{*}, \text{ and } \lim \limits_{t\to \infty} h(t) = h^{*} \coloneqq V^{-1}(v^{*}). 
    \label{v_h_limits}
\end{align}
Moreover, the convergence is exponential: there exist constants \(C>0\) and
\[
\lambda \coloneqq \min\Bigl\{\alpha+\tfrac{\beta}{h_0^2},\ \tfrac{\beta}{(h^{*})^2}\Bigr\}
\]
such that for all \(t\ge0\),
\begin{align}
    |v(t)-v^{*}| + |V(h(t))-v^{*}| \le C\,\mathrm{e}^{-\lambda t}.
    \label{eq:exp_decay_v_V}
\end{align}
In particular, \(h(t)\to h^{*}\) exponentially as \(t\to\infty\).
\label{lem:BE}
\end{lemma}
\begin{proof}
The argument for \eqref{condition_E} is analogous to that for \eqref{condition_B}. Hence, we only present the proof under \eqref{condition_B}.
Assume that the initial datum \((x_0, x_{\ell, 0}, v_0) \in \R^2 \times \R_{\geq 0} \) associated with \cref{eqn: leading,dyn_2_vehicle} is such that \((v_0, V(h_0)) \in B\) and \((v(t), V(h(t))) \in B \coloneqq \{(v,V)\in\R_{\ge0}^2 \mid v^{*}>V>v\}, \forall~ t \in \R_{\geq 0}\).  By system \eqref{dyn_2_vehicle} and \cref{eqn: leading_constant_velocity}, we have, for all \(t \in \R_{\geq 0}\), 
\begin{align*}
    \dot{v}(t)
    &= \alpha\big(V(h(t)) - v(t)\big)
      + \beta \tfrac{v^{*}-v(t)}{h^2(t)}. \\
    \intertext{Since \(\alpha >0\) and because \((v(t),V(h(t)))\in B\) implies \(V(h(t))<v^{*}, \forall t \in \R_{\geq 0}\), we obtain}
    \dot{v}(t) &\le \alpha\big(v^{*}-v(t)\big)
      + \beta \tfrac{v^{*}-v(t)}{h^2(t)}. \\
    \intertext{Moreover, in region \(B\), we have \(h'(t)=v^{*}-v(t)>0\) for all \(t\ge0\), and hence \(h(t)\ge h_0\) for all \(t\ge0\). Therefore,}
    \dot{v}(t) &\le \alpha\big(v^{*}-v(t)\big)
      + \beta \tfrac{v^{*}-v(t)}{h_0^2} = \Bigl(\alpha + \tfrac{\beta}{h_0^2}\Bigr)\bigl(v^{*}-v(t)\bigr).
\end{align*}
In addition, \(\alpha >0\) and \((v(t),V(h(t)))\in B\) implies that 
\begin{align*}
  \dot{v}(t) &= \alpha\big(V(h(t))-v(t)\big)+\beta\tfrac{v^*-v(t)}{h^2(t)} \geq \beta\tfrac{v^*-v(t)}{h^2(t)}\\
  \intertext{by the monotonicity of the optimal velocity function \(V\) and \(v^{*}>V\) in region \(B\), we have \(h(t) < h^* \coloneqq V^{-1}(v^{*})\) for all \(t \in \R_{\geq 0}\), 
}
  \dot{v}(t) & \geq  \beta\tfrac{v^*-v(t)}{(h^{*})^2}.
\end{align*}

Therefore, for all \(t\in\R_{\geq0}\), we have
\begin{equation*}
-\big(\alpha + \tfrac{\beta}{h_0^2}\big)(v^* - v(t)) \leq \tfrac{\dd}{\dd t} (v^* - v(t)) \leq -\tfrac{\beta}{(h^*)^2}(v^* - v(t)),
\end{equation*}
leading us to the following two-sided bound on \(v^* - v\) by Gr\"{o}nwall's inequality: 
\begin{equation}\label{eq:two_sided}
(v^* - v_0) \exp\Big(-\big(\alpha + \tfrac{\beta}{h_0^2}\big)t\Big) \leq v^* - v(t) \leq (v^* - v_0)\exp\Big(-\tfrac{\beta}{(h^*)^2}t\Big), \ \forall t \in \R_{\geq 0}. 
\end{equation}
Notice that in region \(B\), 
\begin{equation}0<v^{*}-V(h(t)) < v^{*}-v(t)
\label{eqn_v_star_V}
\end{equation}
because \(v^* > V(h(t))> v(t) \ \forall t \in \R_{\geq 0}\). 
 Let 
\(
\lambda \coloneqq \min\Bigl\{\alpha+\tfrac{\beta}{h_0^{2}},\ \tfrac{\beta}{(h^{*})^{2}}\Bigr\}.
\)
Then \cref{eq:two_sided,eqn_v_star_V} imply that 
\[
|v(t)-v^{*}| \le (v^{*}-v_0)\,\mathrm{e}^{-\lambda t},
\qquad
|V(h(t))-v^{*}|
\le (v^{*}-v_0)\,\mathrm{e}^{-\lambda t}.
\]
Thus, setting \(C = 2\,(v^{*} - v_0)\), we obtain \cref{eq:exp_decay_v_V}, which implies \(v(t) \to v^{*}\) and \(V(h(t)) \to v^{*}\) as \(t \to \infty\), both of which are exponential with convergence rate \(\lambda\). 
Since \(V\) is invertible, it follows that \(h(t)\to V^{-1}(v^{*}) = h^{*}\) exponentially as  \(t\to\infty\) exponentially. 

\end{proof}

Since the boundaries of regions \(B\) and \(E\) share portions with the boundaries of \(A\), \(C\), \(D\), and \(F\), we treat case~\ref{case:BE_2} jointly with the situations in which the pair \((v,V)\) initially lies in \(A\), \(C\), \(D\), or \(F\) or on their respective boundaries.

\begin{lemma}[behavior of \((v,V)\) in regions \(A\), \(C\), \(D\), and \(F\) and on their boundaries]
\label{lem:AD}
Let the conditions of \cref{lem:BE} hold. 
Suppose that the solution of \cref{dyn_2_vehicle} with initial datum 
\((x_0,x_{\ell,0},v_0)\in \R^2\times\R_{\ge0}\) satisfies
\begin{equation}
\label{eqn_v_V_ACDF}
(v(\hat t),V(h(\hat t)))\in A\cup C\cup D\cup F
\quad\text{or}\quad
(v(\hat t),V(h(\hat t)))\in \partial(A\cup C\cup D\cup F)
\end{equation}
for some \(\hat t<\infty\), where the regions are defined in \cref{eqn: A_B,eqn: C_D,eqn: E_F}. 
Then there exists a \(\tilde t\ge \hat t\) such that
\[
(v(t),V(h(t)))\in A\cup D\cup\{(v^{*},v^{*})\},
\qquad \forall\, t\ge \tilde t.
\]
\end{lemma}

\begin{proof}
Without loss of generality, we redefine the initial time to be \(\hat{t}\), the first time at which the pair \((v(t),V(h(t)))\) enters one of the regions \(A\), \(C\), \(D\), \(F\), or their boundaries.
Let \((x_0, v_0, x_{\ell, 0}) \in \R \times \R_{\geq 0} \times \R\) be given with initial space headway \(h_0 = x_{\ell, 0} - x_0 -l\). We consider several cases separately.
\begin{description}[style=nextline]
\item[Step 0: Starting from \(\{(v^{*}, v^{*})\}\) ]

First, if \((v_0, V(h_0)) = (v^{*}, v^{*})\), then 
\begin{equation}
\dot{v}(0) = \alpha\left(V(h_0)-v_0\right)+\beta\tfrac{v^*-v_0}{h_0^{2}} =0 \text{ and } \dot{h}(0) = v^{*} - v_0 =0.
\label{eqn: equi_star}
\end{equation}
Therefore, \(h(t) = h_0 = h^{*}~ \forall  t \in [0, +\infty)\), which implies that \((v_0, V(h_0)) = (v^{*}, v^{*}) \ \forall t \in [0, +\infty)\).

\item[Step 1: Starting from region \(A\)]

We claim that the region \((A\coloneqq \{(v, V) \in \R_{\geq 0}^2 \mid v^* > v > V\} )\cup \{(v^{*}, v^{*})\}\) in \cref{fig:decomposition_of_regions} is (forward) invariant with respect to the dynamics of system \eqref{dyn_2_vehicle}. 
Indeed, it suffices to investigate the behavior of \(v\) and \(V\) along the boundaries \(\partial A\) of region \(A\),
\begin{align*}
    \partial A = \big\{(v, V)\in \R_{\geq 0}^2 \mid v=v^{*}>V \text{ or } v=V<v^{*} \text{ or } V=0 \wedge v < v^{*} \text{ or } v=V=v^{*}\big\}.
\end{align*}
We have three cases:
\begin{enumerate}
    \item If \(v_0 = v^* > V(h_0)\), then the acceleration of the follower at the initial time will be negative since
\begin{align*}
\dot{v}(0) = \alpha\left(V(h_0)-v_0\right)+\beta\tfrac{v^*-v_0}{h_0^{2}} = \alpha\left(V(h_0)-v_0\right) < 0.
\end{align*} 
Thus, starting from the set \(\{(v, V)\in \R_{\geq 0}^2 \mid v=v^{*}>V\}\), the pair \((v, V)\) enters region \(A\).  

\item If \(v_0 = V(h_0) < v^*\), then the acceleration of the follower at the initial time then satisfies 
\begin{align}
\dot{v}(0) &= \alpha\left(V(h_0)-v_0\right)+\beta\tfrac{v^*-v_0}{h_0^{2}} = \tfrac{\beta}{h_0^{2}}\left(v^*-v_0\right)  \geq V'(h_0)(v^{*} - v_0) \label{acc_A_2} =\tfrac{dV}{dt}(h(t))\vert_{t=0}>0,
\end{align}
which follows from assumption \eqref{assump:beta}.
Hence, \(v\) is increasing faster than \(V\) when the pair \((v, V)\) hits the set \(\{(v, V)\in \R_{\geq 0}^2 \mid \ v=V<v^{*} \}\). Therefore, the ODE system \eqref{dyn_2_vehicle} prevents the pair \((v, V)\) leaving region \(A\) through the boundary \(\{(v, V)\in \R_{\geq 0}^2 \mid  v=V<v^{*} \}\).  

\item If \(V(h_0) =0 \wedge v_0 < v^{*}\), then we have
\begin{align*}
V'(h_0)(v^{*} - v_0)>0.
\end{align*}
\end{enumerate}
Thus, the ODE system \eqref{dyn_2_vehicle} drives the pair \((v, V) \in \{(v, V) \mid V=0 \wedge v < v^{*}\}\) back into region \(A\). 
Combining this with Step~0, we conclude that region \(A \cup \{(v^{*}, v^{*})\}\) is invariant.

\item[Step 2: Starting from region \(B\)]

 We consider the case in which \((v_0, V(h_0)) \in B\) and the pair \((v, V)\) hits part of the boundary of region \(B\), \(\{(v,V) \mid V=v<v^{*}\} \cup \{(v,V) \mid V=v^{*}>v\}\) at some finite time. From Step~1, we know that if the pair \((v,V)\) hits the set \(\{(v,V) \mid V=v<v^{*}\}\), then it enters region \(A\). 
In addition, starting from the set \(\{(v,V) \mid V=v^{*}>v\}\), the pair \((v,V)\) enters region \(C\) since 
\[\dot{V}(h)=V'(h)(v^{*}-v)>0\] 
along the set \(\{(v,V) \mid V=v^{*}>v\}\).

\item[Step 3: Starting from region \(C\)]

 We will show that if \(\left(v_0,V(h_0)\right)\) is in region \(C \coloneqq \{(v, V) \in \R_{\geq 0}^2\mid V>v^* > v\}\), then there exists some \(t_1 \in (0,+\infty)\) such that \(\left(v(t_1),V(h(t_1))\right) \in D \coloneqq \{(v, V) \in \R_{\geq 0}^2 \mid V >v > v^*\}\). 
Assume that the pair \((v_0, V(h_0)) \in C\). First, we claim that \(V(t) > \max\{v^{*}, v(t)\}, \forall t \in [0, \infty)\). In fact, along the part of boundary of region \(C\), \(\{(v,V) \in \R_{\geq 0}^2 \mid V=v^{*}>v\}\), 
\[\dot{V}(h) = V'(h)(v^{*}-v)>0,\]
which implies that \(V(t) > v^{*}\)  \(\forall t \in [0, \infty)\). In other words, the pair \((v, V)\) can never enter region \(B\). Therefore, it either stays in region \(C\) or enters region \(D\). We will prove in Step~\(4\) that region \(D \cup \{(v^{*}, v^{*})\}\) is invariant. In either case, the pair \((v,V)\) starting from region \(C\) satisfies \(V(t) > v(t) \ \forall t \in [0, \infty)\).

Now we argue that \(\left(v(t_1),V(h(t_1))\right) \in D \coloneqq \{(v, V) \in \R_{\geq 0}^2 \mid V >v > v^*\}\) at some \(t_1 \in (0, +\infty)\) by contradiction. Assume that \(v(t) < v^*\) for all \(t \geq 0\). Then for all \(t \in [0, \infty)\), 
 \[\tfrac{\dd}{\dd t} (v^* - v(t)) = -\alpha \left(V(h(t))-v(t)\right) - \beta\tfrac{v^* - v(t)}{h^2(t)} \leq -\alpha\left(V(h(t))-v(t)\right) \leq -\alpha(v^* - v(t)),\]
 where the last inequality follows from \(V(t)>v^{*}, \forall t \in [0, \infty)\). From the assumption that \(v(t) < v^*\) for all \(t \geq 0\) and Gr\"{o}nwall's inequality, we have  
\begin{equation}
0 < v^* - v(t) \leq \left(v^* - v_0\right)\exp(-\alpha t)  \ \forall t \geq 0.
\label{inq_v_v}
\end{equation}
Note also that \((\ref{inq_v_v})\) implies that there exists some constant \(C>0\) such that 
\[\int_{0} ^\infty \left(v^* - v(t)\right)\dd t = \int_{0} ^\infty h'(t)\dd t =C. \]
From the assumption that \(v(t) < v^*\) for all \(t \geq 0\), we obtain
\[C=\int_{0} ^\infty h'(s)\dd s > \int_{0} ^t h'(s)\dd s = h(t) - h_0 \geq 0,  \forall t \geq 0,\]
which implies  \(h_0 \leq h(t) \leq h_0 + C\) for all \(t \geq 0\) and
\begin{align*}
\dot{v}(t) = \alpha\left(V(h(t))-v(t)\right)+\beta\tfrac{v^*-v(t)}{h^{2}(t)} > \alpha\left(V(h(t))-v(t)\right) > \alpha\left(V(h(t))-v^{*}\right),  \forall t \geq 0,
\end{align*}
which leads to
\[v(t) - v_0 = \int_{0}^{t} \dot{v}(s)\dd s \geq \alpha\int_{0}^{t} \left(V(h(s)) - v^*\right) \dd s 
 \geq \alpha\left(V(h_0) - v^*\right)(t - 0),  \forall t \geq 0.\]
Consequently, taking \[t_1 = \tfrac{v^* - v_0}{\alpha\left(V(h_0) - v^*\right)}\] leads us to \(v(t_1) \geq v^*\), contradicting the assumption that \(v(t) < v^*\) for all \(t \geq 0\).
 
\item[Step 4: Starting from region \(D\)]

We claim that region \((D \coloneqq \{(v,V) \in \R_{\geq 0}^2 \mid V>v>v^{*}\}) \cup \{(v^{*}, v^{*})\}\) in \cref{fig:decomposition_of_regions}, is (forward) invariant with respect to the dynamics of the system \eqref{dyn_2_vehicle}. 
We start by investigating the behavior of the pair \((v,V)\) along the boundary \(\partial D\) of region \(D\), where 
\begin{align*}
    \partial D = \big\{(v, V)\in \R_{\geq 0}^2 \mid v=v^{*}<V \text{ or } v=V>v^{*} \text{ or } v=V=v^{*}\big\}.
\end{align*}
We consider three cases:
\begin{enumerate}
\item As proved in Step~1, if the pair \((v, V)\) starts at \((v^{*}, v^{*})\), then it stays there because of \cref{eqn: equi_star}. 
\item If \(v_0 = v^{*} < V(h_0)\), then the acceleration of the leader at the initial time is positive. Indeed, 
\begin{align*}
\dot{v}(0) = \alpha\left(V(h_0)-v_0\right)+\beta\tfrac{v^*-v_0}{h_0^2} = \alpha\left(V(h_0)-v_0\right) >0.
\end{align*}
This implies that following the dynamics of system \eqref{dyn_2_vehicle}, the pair \((v, V)\) cannot pass from region \(D\) to region \(C\). 
\item If \(v_0=V(h_0)>v^{*}\), then by assumption \eqref{assump:beta} and the fact that \(\dot{h} = v^{*}-v\),
\begin{align*}
\dot{v}(0)& =  \alpha\left(V(h_0)-v_0\right)+\beta\tfrac{v^*-v_0}{h_0^2} = \beta\tfrac{v^*-v_0}{h_0^2}\\
 &< V'(h_0) (v^{*} - v_0) = \dot{V}(h_0) <0.
\end{align*}
\end{enumerate}
Thus, \(v\) is decreasing faster than \(V\) at the initial time if  \(v_0=V(h_0)>v^{*}\). Therefore, the pair \((v,V) \in \{(v, V) \in \R_{\geq 0}^2 \mid v=V > v^{*}\}\) can only enter region \(D\). 
Hence, region \(D \cup \{(v^{*}, v^{*})\}\) is invariant.

\item[Step 5: Starting from region \(E\)] 

Similar to the argument in Step~2, we consider the case in which \((v_0, V(h_0)) \in E\coloneqq \{(v, V)\in \R_{\geq 0}^2 \mid v > V >v^*\}\) and \((v,V)\) hits the part of the boundary of region \(E\), \(\{(v,V)\in \R_{\geq 0}^2  \mid v=V>v^{*}\} \cup \{(v,V) \in \R_{\geq 0}^2 \mid V=v^{*}<v\}\), at some finite time. 
In Step~4, we have shown that if \((v,V)\) hits the set \(\{(v,V)\in \R_{\geq 0}^2  \mid v=V>v^{*}\), then it enters region \(D\). Furthermore, if \((v,V) \) hits the set \(\{(v,V)\in \R_{\geq 0}^2 \mid V=v^{*}<v\}\), then it enters region \(F\). This is because along the set \(\{(v,V)\in \R_{\geq 0}^2  \mid V=v^{*}<v\)\}, 
\[\dot{V}(h) = V'(h)(v^{*}-v)<0.\]

\item[Step 6: Starting from region \(F\)]

We argue that if \(\left(v_0,V(h_0)\right) \in F \coloneqq \{(v, V) \in \R_{\geq 0}^2\mid v>v^* > V\}\), then there exists some \(t_1 \in (0,+\infty)\) such that \(\left(v(t_1),V(h(t_1))\right) \in A \coloneqq \{(v, V) \in \R_{\geq 0}^2 \mid   v^*>v>V\}\). This 
argument is similar to the proof in Step~3. 
Assume that \(\left(v_0,V(h_0)\right) \in F\). Notice that \(V(h(t)) < \min\{v^{*}, v(t)\} \ \forall t \in [0, \infty)\). Indeed, starting from a part of the boundary of region \(F\), \(\{(v,V) \in \R_{\geq 0}^2 \mid v>V=v^{*}\}\), the pair \((v,V)\) can never enter region \(E\). This is because if \((v(\tilde{t}), V(\tilde{t})) \in \{(v,V) \in \R_{\geq 0}^2 \mid v>V=v^{*}\} \text{ for some } \tilde{t} \in \R_{\geq 0}\), then 
\[\dot{V}(h(\tilde{t})) = V'(h(\tilde{t}))(v^{*}-v(\tilde{t}))<0.\]
Recall from Step~1 that region \(A \cup \{(v^{*}, v^{*})\}\) is invariant. Therefore, any pair \((v,V)\) starting from region \(F\) can never cross either of the lines \(\{(v,V)\in \R_{\geq 0}^2 \mid V=v^{*}\}\) and \(\{(v,V)\in \R_{\geq 0}^2 \mid v=V\}\), which implies \(V(h(t)) < v^{*} \text{ and } V(h(t))<v(t) \ \forall t \in [0, \infty)\). 

Now we need to show that there exists some \(t_1 \in (0, +\infty)\) such that \(v^{*}>v(t_1)\). We will again argue this by contradiction and assume that \(v(t)>v^{*} \text{ for } \forall t \in [0, \infty)\). Then 
\begin{align*}
\tfrac{\dd}{\dd t} (v(t) - v^{*}) = & \alpha \left(V(h(t))-v(t)\right) + \beta\tfrac{v^* - v(t)}{h^2(t)} \\
\intertext{because of the assumption that \(v(t)>v^{*} \ \forall t \in [0, \infty)\) and \(\beta >0\), and}
\tfrac{\dd}{\dd t} (v(t) - v^{*}) \leq & \alpha(V(h(t))-v(t)),\\
\intertext{which follows from the fact that \(V(h(t)) < v^{*} \ \forall t \geq 0\). We then have}
\tfrac{\dd}{\dd t} (v(t) - v^{*}) \leq & \alpha(v^{*}-v(t))= -\alpha (v(t) - v^{*}).
\end{align*}
Using Gr\"{o}nwall's inequality, we obtain 
\begin{equation}
0 < v(t)-v^* \leq (v_0 - v^{*})\exp{(-\alpha t )}  \ \forall t \in\R_{\geq0}.
\end{equation}
Thus, there exists some constant \(C>0\) such that 
\[\int_{0} ^\infty \left(v(t) - v^{*}\right)\dd t = -\int_{0} ^\infty h'(t)\dd t =C. \]
Note that \(h'(t) = v^{*}-v(t) <0\) because of the assumption that \(v(t) > v^* \ \forall t \in \R_{\geq0}\). Thus  
\[-C=\int_{0} ^\infty h'(s)\dd s < \int_{0} ^t h'(s)\dd s = h(t) - h_0 \leq 0,\]
which implies  \(h_0 -C\leq h(t) \leq h_0\) for all \(t \geq 0\) and
\begin{align*}
\dot{v}(t) = \alpha\left(V(h(t))-v(t)\right)+\beta\tfrac{v^*-v(t)}{h^{2}(t)} <\alpha\left(V(h(t))-v(t)\right) < \alpha\left(V(h(t))-v^{*}\right),  \forall t\in\R_{\geq0}.
\end{align*}
Therefore, 
\[v(t) - v_0 = \int_{0}^{t} \dot{v}(s)\dd s \leq  \alpha\int_{0}^{t} \left(V(h(s)) - v^*\right) \dd s 
 \leq \alpha\left(V(h_0) - v^*\right)(t - 0),\quad \forall t \geq 0.\]
 Letting \[t_1 = \tfrac{v^* - v_0}{\alpha\left(V(h_0) - v^*\right)},\] we have \(v(t_1) \leq v^*\), which contradicts the assumption that \(v(t) > v^*\) for all \(t \geq 0\).
\end{description}
\end{proof}

Now we are ready to introduce the following energy function associated with \cref{dyn_2_vehicle}: 
\begin{equation}
\mathcal{E}(t) \coloneqq \tfrac{1}{2}\big(V(h(t)) - v^*\big)^2,\   t\in\R_{\geq0}.
\label{eqn: Energy_1}
\end{equation}
This plays a central role in characterizing the long-term behavior of the follower in the two-vehicle setting. In particular, \(\mathcal{E}(t)\to 0\) as \(t\to\infty\); see \cref{fig:E_F}.
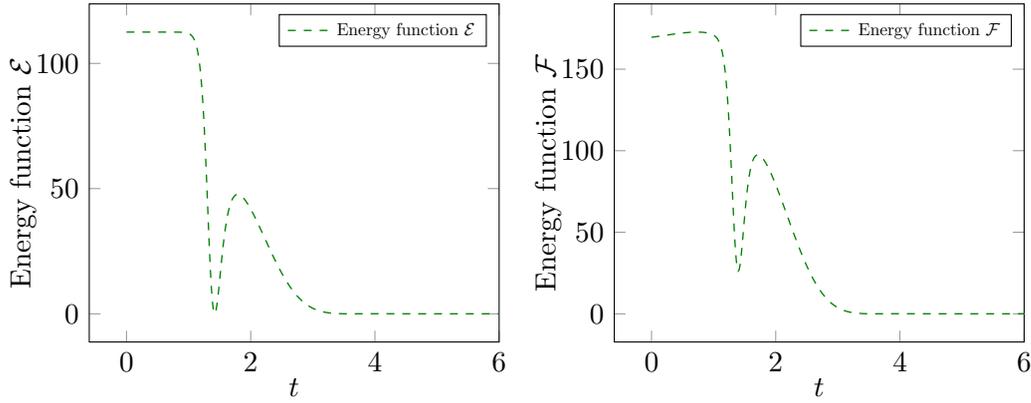
\begin{figure}
\centering
    \begin{tikzpicture}
         \begin{axis}[
             width=0.45\textwidth,
             xlabel={ $t$ },
             ylabel={Energy function \(\mathcal{E}\)}, 
             ylabel style={yshift=-10pt},
             xlabel style={yshift=5pt},
                 legend pos={north east, xmax=6,
                 legend style={nodes={scale=0.6, transform shape}}}
               ]
         \addplot[color=darkgreen,line width=0.5pt, mark=none, dashed]  table [y index=1, x index=0, col sep=comma] {Data/Energy_E_star.csv};
           \addlegendentry{Energy function \(\mathcal{E}\)};
         \end{axis}
     \end{tikzpicture}
     \begin{tikzpicture}
         \begin{axis}[
             width=0.45\textwidth,
             xlabel={ $t$ },
             ylabel={Energy function \(\mathcal{F}\)}, 
            ylabel style={yshift=-10pt},
            xlabel style={yshift=5pt},
                 legend pos={north east, xmax=6,
                 legend style={nodes={scale=0.6, transform shape}}}
               ]
         \addplot[color=darkgreen,line width=0.5pt, mark=none, dashed]  table [y index=1, x index=0, col sep=comma] {Data/Energy_F_star.csv};
           \addlegendentry{Energy function \(\mathcal{F}\)};
         \end{axis}
     \end{tikzpicture}
 \caption{\textbf{Left}: Evolution of the energy $\mathcal E$. \textbf{Right}: Evolution of the energy \(\mathcal{F}\).}
\label{fig:E_F}
\end{figure}

\begin{theorem}[energy decay and convergence to equilibrium]
\label{thm_E_equilibrium}
Let \(\mathcal{E}\) be the energy function defined in \cref{eqn: Energy_1}. Suppose that the dynamics of the leader and follower are governed by \cref{eqn: leading,dyn_2_vehicle}, respectively, under the initial condition \((x(0), x_{\ell}(0), v(0)) = (x_0, x_{\ell, 0}, v_0) \in \R^2\times \R_{\geq 0}\) with \(h_0 \coloneqq x_{\ell,0}-x_0-l>0\). Assume that \cref{assump:beta} is fulfilled. Then, the follower's velocity and space headway approach an equilibrium as \(t \to \infty\):
\begin{align*}
\lim\limits_{t \to \infty} v(t) = v^*\text{ and } \lim\limits_{t \to \infty}h(t) = h^{*},
\end{align*}
where \(v^{*}\) and \(h^{*}\) are as defined in \cref{def: equilibrium}. In addition, 
\[
\lim_{t\to\infty}\mathcal{E}(t)=0.
\]
\end{theorem}
\begin{proof}
    We consider the following two cases: 
    \begin{enumerate}
        \item If the initial condition \((x_0,x_{\ell,0},v_0)\in\R^2\times\R_{\ge0}\) is such that the pair \((v,V)\) remains in region \(B\) or \(E\) for all time (i.e., \cref{condition_B} or \cref{condition_E} holds), then, by \cref{lem:BE}, we have
    \[
    \lim_{t\to\infty} V(h(t)) = v^{*},
    \]
    with exponential convergence rate
    \[
    \lambda \coloneqq \min\Bigl\{\alpha+\tfrac{\beta}{h_0^2},\ \tfrac{\beta}{(h^{*})^2}\Bigr\},
    \qquad \text{where } h^{*}=V^{-1}(v^{*}).
    \]
    Consequently, \(\mathcal{E}(t)\to 0\) exponentially as \(t\to\infty\).
       \item If the initial condition \((x_0,x_{\ell,0},v_0)\in\R^2\times\R_{\ge0}\) is such that the trajectory \((v(t),V(h(t)))\) reaches one of the regions \(A\), \(C\), \(D\), or \(F\) or their boundaries in finite time (i.e., \cref{eqn_v_V_ACDF} holds), then, by \cref{lem:AD}, there exists a \(\tilde{t}>0\) such that
\[
(v(t),V(h(t)))\in A\cup D\cup\{(v^{*},v^{*})\},
\quad \forall\, t\ge \tilde t.
\]
Notice that for \(\forall t \in \R_{\geq 0}\), 
\begin{align*}
\dot{\mathcal{E}}(t)&=  \big(V(h(t)) - v^{*}\big)V'(h(t))\dot{h}(t)  \\
&=\big(V(h(t))-v^*\big)V'(h(t))(v^{*}-v(t)) \\
&=  \big(V(h(t))-v(t)+v(t) -v^*\big)V'(h(t))(v^{*}-v(t)) \\
&= V'(h(t))\big(v(t)-V(h(t))\big)(v(t)-v^*)-V'(h(t))(v(t)- v^*)^2. 
\end{align*}
Recall that \(V'>0\) and 
\begin{align*}
A = & \{(v, V) \in \R_{\geq 0}^2 \mid v^* > v > V\} ~~\text{and}~~ D =\{(v, V)\in \R_{\geq 0}^2 \mid V > v > v^*\}. 
\end{align*}
Thus,
\begin{align*}
\lim \limits_{t \to \infty} \dot{\mathcal{E}}(t) & =  \lim \limits_{t \to \infty} \big(V'(h(t))\left(v(t)-V(h(t))\right)\left(v(t)-v^*\right)-V'(h(t))(v(t)- v^*)^2 \big)\\
& \leq \lim \limits_{t \to \infty} -V'(h(t))(v(t)- v^*)^2 \leq 0, \quad  \forall t \in \R_{\geq 0}.
\end{align*}
As a result, the limit \(\mathcal E_\infty \coloneqq \lim\limits_{t \to \infty} \mathcal{E}(t) = \lim\limits_{t \to \infty} \tfrac{1}{2} \big(V(h(t)-v^{*})\big)^2\geq 0\) exists because of the monotonicity of  \(\mathcal{E}\) and its boundedness from below by zero. This implies that \(\lim\limits_{t \to \infty} \dot{\mathcal{E}}(t) =0\) and that \(\lim\limits_{t \to \infty} | V(h(t)) - v^{*}| \) exists. Therefore, 
\[\lim\limits_{t \to \infty} v(t) = v^{*}. \]
Furthermore, 
\begin{equation}
\lim\limits_{t \to \infty} | V(h(t)) - v^{*}| = \begin{cases}
\lim\limits_{t \to \infty} V(h(t)) - v^{*}  & \text{ if } \lim \limits_{t \to \infty }(v(t), V(h(t))) \in D \cup \{(v^{*}, v^{*})\} \\
\lim \limits_{t \to \infty} -V(h(t)) + v^{*}   &  \text{ if } \lim \limits_{t \to \infty } (v(t), V(h(t))) \in A \cup \{(v^{*}, v^{*})\} .
\end{cases}
\label{eqn: energy}
\end{equation}
Because of the monotonicity of the optimal velocity function \(V\) as defined in \cref{eq:assumption_V}, we also deduce the existence of \(h_\infty \coloneqq \lim\limits_{t \to \infty} h(t)\).  Note that 
\begin{align*}
    \lim \limits_{t \to \infty } \dot{v}(t)= \lim \limits_{t \to \infty} \Big(\alpha\big(V(h(t))-v(t)\big)+\beta\tfrac{v^*-v(t)}{h^2(t)}\Big)
\end{align*}
exists because \(\lim\limits_{t \to \infty} v(t) = v^{*}\) and \(\lim\limits_{t \to \infty} h(t) = h_{\infty}>0\). Furthermore, 
\begin{align}
  0=\lim\limits_{t \to \infty} \dot{v}(t) & = \lim\limits_{t \to \infty} 
 \Big(\alpha\big(V(h(t))-v(t)\big)+\beta\tfrac{v^*-v(t)}{h^2(t)}\Big) \label{eqn_0_acc}\\
 &=\lim\limits_{t \to \infty} 
 \alpha\big(V(h(t))-v(t)\big). 
 \end{align}
Consequently, \[\lim\limits_{t \to \infty} V(h(t)) = V(h_\infty) =  v^*,\] and hence \(\mathcal{E}(t) \to 0\) as \(t \to \infty\). In addition, 
from the monotonicity of the optimal velocity function \(V\), we have $h_\infty = V^{-1}(v^{*})=h^*$, and thus \(\lim \limits_{t \to \infty}h(t) = h^{*}\).  
\end{enumerate}
\end{proof}
\subsection{Alternative approach to the long-term convergence}

It is also possible to establish a long-term convergence guarantee for the solution of the Bando-FtL system \cref{dyn_2_vehicle} by employing a more general energy function:
\begin{equation}\label{eq:F}
  \mathcal{F}(t) \coloneqq \tfrac{1}{2}\left(V(h(t)) - v^*\right)^2 + \tfrac{1}{2}\left(v(t) - v^*\right)^2 + \tfrac{1}{2}\left(V(h(t)) - v(t)\right)^2 ,\quad \forall t \geq 0.
\end{equation}
As shown in \cref{fig:E_F}, \(\lim \mathcal{F}(t) \to 0\) as \(t \to \infty\).
In addition to the assumption \cref{assump:beta} about the size of the ``follow-the-leader'' effect (measured by $\beta$), we also impose the following assumption on the coefficient $\alpha$ in \cref{dyn_2_vehicle}.
\begin{assumption}
\label{assump_2}
The coefficient $\alpha\in\R_{>0}$ in the Bando-FtL acceleration \cref{defi:ACC} satisfies
\begin{equation}\label{assump:alpha}
\max\left\{V'(h),\tfrac{\beta}{h^2}\right\} < \alpha < \tfrac{V'(h)}{2} + \tfrac{\beta}{h^2} \quad \text{for all $h_{\min} \leq h\leq h_{\max}$},
\end{equation}
where \(h_{\min} \text{ and } h_{\max}\) are as in \cref{eqn_h_min,eqn_h_max}, respectively. 
\end{assumption}

\begin{remark}
The conditions \cref{assump:beta} in \cref{assump_1,assump:alpha} in \cref{assump_2} together imply that
\begin{equation}\label{assump:combined}
\tfrac{\beta}{h^2} < \alpha < \tfrac{V'(h)}{2} + \tfrac{\beta}{h^2}
\quad \text{for all } h_{\min} \le h \le h_{\max}.
\end{equation}
Although the upper and lower bounds in \eqref{assump:combined} impose a restrictive constraint on \(\alpha\), the admissible range is not empty. Indeed, for each \(h\in[h_{\min},h_{\max}]\) the interval
\[
\Bigl(\tfrac{\beta}{h^2},\ \tfrac{V'(h)}{2}+\tfrac{\beta}{h^2}\Bigr)
\]
has positive length \(\tfrac{V'(h)}{2}\), and thus a feasible \(\alpha\) exists whenever \(V'(h)>0\) on \([h_{\min},h_{\max}]\).
\end{remark}

Considering the assumption \cref{assump:alpha} about the size of the coefficient $\alpha$ together with \cref{assump:beta} about the coefficient $\beta$, we can promote the qualitative convergence result reported in \cref{thm_E_equilibrium} to a quantitative exponential convergence guarantee.
\begin{theorem}
[quantitative long-term behavior of the follower]
\label{thm: main2}
Assume that the dynamics of the leader and follower are described by \cref{eqn: leading,dyn_2_vehicle} and that the optimal velocity function \(V\) satisfies \cref{eq:assumption_V}. 
Let \cref{assum_initial_following,assum_leading_const_v,assump_1,assump_2} hold. 
Then \(\mathcal{F}(t)\to 0\) exponentially as \(t\to\infty\). In addition, the follower's velocity and space headway converge exponentially to their equilibrium values \(v^{*}\) and \(h^{*}\) as \(t\to\infty\). 
More precisely, for all \(t\in\R_{\ge0}\),
\begin{equation}\label{eq:v_decay}
|v(t) - v^*|
\leq 
\bigl(|V(h_0) - v^*| + |v_0 - v^*| + |V(h_0) - v_0|\bigr)
\exp\!\left(-\tfrac 12\left(\alpha - \tfrac{\beta}{h^2_{\min}}\right)t\right)
\end{equation}
and
\begin{equation}\label{eq:h_decay}
|h(t) - h^*|
\leq 
2\tfrac{|V(h_0) - v^*| + |v_0 - v^*| + |V(h_0) - v_0|}{\alpha - \beta/h^2_{\min}}
\exp\!\left(-\tfrac 12\left(\alpha - \tfrac{\beta}{h^2_{\min}}\right)t\right).
\end{equation}
\end{theorem}
\begin{proof}
The time evolution of $\mathcal{F}$ along the solution of the system \cref{dyn_2_vehicle} can be evaluated as follows:
\begin{align*}
\tfrac{\dd}{\dd t} \mathcal{F}(t) &=  \left(V(h(t))-v^*\right) V'(h(t)) (v^{*}-v(t)) + (v(t)-v^*) \left(\alpha\left(V(h(t))-v(t)\right) + \beta \tfrac{v^* - v(t)}{h^2(t)}\right) \\
&\qquad + \left(V(h(t))-v(t)\right) \left(V'(h(t)) (v^* - v(t)) - \alpha \left(V(h(t))-v(t)\right) - \beta\tfrac{v^* - v(t)}{h^2(t)}\right)\\
&= \left(V(h(t))-v^*\right) V'(h(t)) \left(v^*-V(h(t))+V(h(t))-v(t)\right) \\
&\qquad -\tfrac{\beta}{h^2(t)} (v^*-v(t))^2 + \alpha \left(V(h(t))-v(t)\right) (v(t)-v^*) \\
&\qquad + \left(V(h(t))-v(t)\right) \left(V'(h(t)) (v^* - v(t)) - \alpha \left(V(h(t))-v(t)\right) - \beta \tfrac{v^* - v(t)}{h^2(t)}\right). \\
\intertext{After rearranging terms, it follows that }
\tfrac{\dd}{\dd t} \mathcal{F}(t) &= -V'(h(t)) \left(V(h(t))-v^*\right)^2 - \tfrac{\beta}{h^2(t)} (v^*-v(t))^2 - \alpha \left(V(h(t))-v(t)\right)^2 \\
&\qquad + V'(h(t)) \left(V(h(t))-v^*\right) \left(V(h(t))-v(t)\right) \\
&\qquad + \left(\alpha - V'(h(t)) + \tfrac{\beta}{h^2(t)}\right) \left(V(h(t))-v(t)\right) (v(t)-v^*). \\
\intertext{Using \(a \leq |a|\) and \(ab \leq \tfrac{a^2 + b^2}{2}, \forall a, b \in \R\), we have}
\tfrac{\dd}{\dd t} \mathcal{F}(t) &\leq -V'(h(t)) \left(V(h(t))-v^*\right)^2 - \tfrac{\beta}{h^2(t)} (v^*-v(t))^2 - \alpha \left(V(h(t))-v(t)\right)^2 \\
&\qquad + V'(h(t)) \tfrac{\left(V(h(t))-v^*\right)^2}{2} + V'(h(t)) \tfrac{\left(V(h(t))-v(t)\right)^2}{2} \\
&\qquad + \left|\alpha - V'(h(t)) + \tfrac{\beta}{h^2(t)}\right| \tfrac{(V(h(t))-v(t))^2}{2} + \left|\alpha - V'(h(t)) + \tfrac{\beta}{h^2(t)}\right| \tfrac{(v(t)-v^*)^2}{2} \\
&= -V'(h(t)) \tfrac{\left(V(h(t))-v^*\right)^2}{2} + \left(-\tfrac{2\beta}{h^2(t)}+\left|\alpha - V'(h(t)) + \tfrac{\beta}{h^2(t)}\right|\right) \tfrac{(v(t)-v^*)^2}{2} \\
&\qquad + \left(-2 \alpha + V'(h(t)) + \left|\alpha - V'(h(t)) + \tfrac{\beta}{h^2(t)}\right| \right) \tfrac{(V(h(t))-v(t))^2}{2},\\
\intertext{where the last equality was obtained by combining like terms. By assumption \cref{assump:alpha}, \(\beta >0\), and \(\alpha - V'(h(t)) + \tfrac{\beta}{h^2(t)}>0\),}
\tfrac{\dd}{\dd t} \mathcal{F}(t) &\leq \max\left\{-V'(h(t)), -\tfrac{2\beta}{h^2(t)}+\left|\alpha - V'(h(t)) + \tfrac{\beta}{h^2(t)}\right|, \right.\\
& \quad \left. -2 \alpha + V'(h(t)) + \left|\alpha - V'(h(t)) + \tfrac{\beta}{h^2(t)}\right|\right\} \mathcal{F}(t) \\
&= \max\left\{-V'(h(t)), \alpha - V'(h(t)) - \tfrac{\beta}{h^2(t)}, -\alpha + \tfrac{\beta}{h^2(t)}\right\}\mathcal{F}(t). \\
\intertext{Because of \cref{assump:alpha}, \(-\alpha + \tfrac{\beta}{h^2(t)} > \alpha - V'(h(t)) - \tfrac{\beta}{h^{2}(t)} > - V'(h(t))\), and thus }
\tfrac{\dd}{\dd t} \mathcal{F}(t) &\leq -\left(\alpha - \tfrac{\beta}{h^2(t)}\right) \mathcal{F}(t) \leq -\left(\alpha - \tfrac{\beta}{h^2_{\min}}\right) \mathcal{F}(t).
\end{align*}
Consequently, we deduce that $\mathcal{F}(t) \leq \mathcal{F}(0) \exp\left\{-\left(\alpha - \tfrac{\beta}{h^2_{\min}}\right)t\right\}$ for all $t\geq 0$, from which the estimate \eqref{eq:v_decay} follows. Now, as $h(t) \xrightarrow{t\to \infty} h^*$ by \cref{thm_E_equilibrium}, we have for \(t\in\R_{\geq0}\) that $h(t) = h_0 + \int_0^t (v^* - v(s)) \dd s$ and $h^* = h_0 + \int_0^\infty (v^* - v(s)) \dd s$. Therefore, the estimate \eqref{eq:h_decay} can be derived by noticing that 
\begin{align*}
|h(t) - h^*| & \leq \int_t^\infty\!\! |v^* - v(s)|  \dd s \\
& \leq \big(|V(h_0) - v^*| + |v_0 - v^*| + |V(h_0) - v_0|\big) \int_t^\infty\!\!\! \exp\Big(-\tfrac 12\left(\alpha - \tfrac{\beta}{h^2_{\min}}\right)s\Big) \dd s. 
\end{align*}
\end{proof}

\begin{remark}
The quantitative estimates \eqref{eq:v_decay} and \eqref{eq:h_decay} established in \cref{thm: main2} do not represent straightforward refinements of the qualitative convergence results reported in \cref{thm_E_equilibrium}, as they rely on the additional assumption \eqref{assump:alpha} imposed on the coefficient $\alpha$.
\end{remark}
\subsection{Long-term behavior of the Bando-FtL model with \texorpdfstring{$N+1$}{\textit{N} + 1} vehicles}

In this subsection, we study the long-term behavior of the Bando-FtL model in the general case, with $N+1$ vehicles for $N \geq 2$.  The long-term behavior in the five-vehicle case is illustrated in \cref{fig_5_cars_star}. The key ingredient is the propagation of uniform bounds on the space headways $h_i = x_{i-1} - x_i - l$ for all $t\geq 0$ and all $i \in \{2,\ldots,N+1\}$ proved in \cref{cor_well_posed_N_cars}. Our main result is the following long-term convergence guarantee.

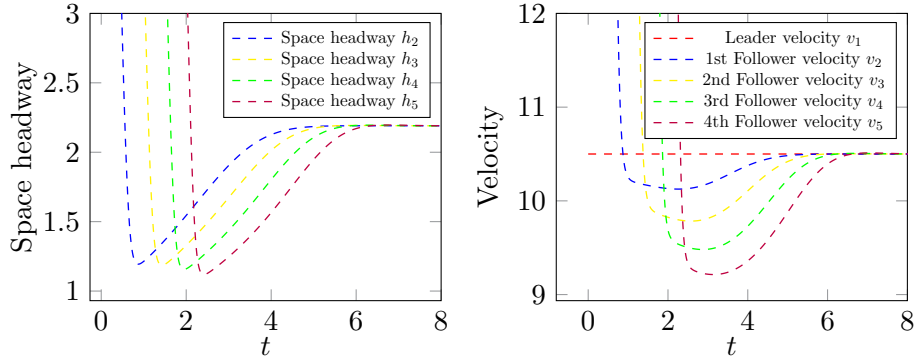
\begin{figure}
\centering
    \begin{tikzpicture}
         \begin{axis}[
             width=0.4\textwidth,
             xlabel={ $t$ },
             ylabel={Space headway}, 
             xlabel style={yshift=5pt},
             ylabel style={yshift=-10pt},
                 legend pos={north east,ymax=3, xmax=8, 
                 legend style={nodes={scale=0.6, transform shape}}}
               ]
         \addplot[color=blue,line width=0.5pt, mark=none, dashed]  table [y index=1, x index=0, col sep=comma] {Data/Follower_2_headway_star.csv};
         \addplot[color=yellow,line width=0.5pt, mark=none, dashed]  table [y index=1, x index=0, col sep=comma] {Data/Follower_3_headway_star.csv};
         \addplot[color=green,line width=0.5pt, mark=none, dashed]  table [y index=1, x index=0, col sep=comma] {Data/Follower_4_headway_star.csv};
         \addplot[color=purple,line width=0.5pt, mark=none, dashed]  table [y index=1, x index=0, col sep=comma] {Data/Follower_5_headway_star.csv};
           \addlegendentry{Space headway \(h_2\)};
           \addlegendentry{Space headway \(h_3\)};
           \addlegendentry{Space headway \(h_4\)};
           \addlegendentry{Space headway \(h_5\)};
         \end{axis}
     \end{tikzpicture}
    \begin{tikzpicture}
        \begin{axis}[
            width=0.4\textwidth,
            xlabel={$t$},
            ylabel={Velocity},
            xlabel style={yshift=5pt},
            ylabel style={yshift=-10pt},
            legend pos={north east,ymax=12, xmax=8,
            legend style={nodes={scale=0.6, transform shape}}}
                 ]
        \addplot[color=red,line width=0.5pt, mark=none, dashed]  table [y index=1, x index=0, col sep=comma] {Data/Leader_1_velocity_star.csv};
        \addplot[color=blue,line width=0.5pt, mark=none, dashed]  table [y index=1, x index=0, col sep=comma] {Data/Follower_2_velocity_star.csv};
        \addplot[color=yellow,line width=0.5pt, mark=none, dashed]  table [y index=1, x index=0, col sep=comma] {Data/Follower_3_velocity_star.csv};
        \addplot[color=green,line width=0.5pt, mark=none, dashed]  table [y index=1, x index=0, col sep=comma] {Data/Follower_4_velocity_star.csv};
        \addplot[color=purple,line width=0.5pt, mark=none, dashed]  table [y index=1, x index=0, col sep=comma] {Data/Follower_5_velocity_star.csv};
    \addlegendentry{Leader velocity \(v_{1}\)};
         \addlegendentry{\(1\)st Follower velocity \(v_2\)};
         \addlegendentry{\(2\)nd Follower velocity \(v_3\)};
         \addlegendentry{\(3\)rd Follower velocity \(v_4\)};
         \addlegendentry{\(4\)th Follower velocity \(v_5\)};
        \end{axis}
    \end{tikzpicture}
 \caption{\textbf{Left}: The four followers' space headways.  \textbf{Right}: The five vehicle's velocities.}
 \label{fig_5_cars_star}
\end{figure}

\begin{theorem}[long-term behavior of the Bando-FtL model for \(N+1\) vehicles]
\label{thm:generalization}
Under the conditions of \cref{cor_well_posed_N_cars}, if assumptions~\eqref{assump:beta} and~\eqref{assump:alpha} hold with 
$h_{N+1,\min} = \min\limits_{ 2 \leq i \leq N+1} h_{i,\min}$ replacing \(h_{\min}\), then the velocities and spatial headways of the followers asymptotically converge to a common equilibrium as \(t \to \infty\). In other words, for each $i \in \{2,\ldots, N+1\}$, it holds that
\begin{align*}
\lim\limits_{t \to \infty} v_i(t) = v^* ~\text{ and }~ \lim\limits_{t \to \infty} h_i(t) = h^{*} \coloneqq V^{-1}(v^*).
\end{align*}
\end{theorem}

\begin{proof}
As stated in the proof of \cref{cor_well_posed_N_cars}, the causal structure of the coupled dynamical system \cref{system_dyn_new} implies that the evolution of $\left\{(x_k(t),v_k(t))\right\}_{t\geq 0}$ for $1 \leq k \leq i$ does not depend on the evolution of $\left\{(x_m(t),v_m(t))\right\}_{t\geq 0}$ for $i < m \leq N+1$. Therefore, the state $(x_i,v_i)$ of the $i$th vehicle affects the state $(x_{i+1},v_{i+1})$ of the $(i+1)$th vehicle but not vice versa. This allows us to assume without loss of generality that there are $N = 2$ followers since the general scenario with $N \geq 3$ followers can be handled in a similar manner via an induction argument.

Thanks to \cref{thm_E_equilibrium,thm: main2}, we know that
\begin{equation*}
\lim\limits_{t \to \infty} v_2(t) = v^* ~\text{ and }~ \lim\limits_{t \to \infty} h_2(t) = \lim\limits_{t \to \infty} \left(x_1(t) - x_2(t) - l\right) = h^{*}.
\end{equation*}
Therefore, if we define $\sigma_2(t) \coloneqq v_2(t) - v^*$ for all $t\geq 0$, then we can write $v_2(t) = v^* + \sigma(t)$ with $\lim \limits_{t \to \infty} \sigma_2(t) = 0$. To establish the above convergence results for the third vehicle, with velocity and space–headway pair $\left\{(v_3(t),h_3(t))\right\}_{t\geq 0}$, we use the following natural adaptation of the energy function $\mathcal{F}$ from \eqref{eq:F}, which was employed in the proof of \cref{thm: main2} in the case of a single follower:
\begin{equation}\label{eq:F3}
  \mathcal{F}_3(t) \coloneqq \tfrac{1}{2}\left(V(h_3(t)) - v^*\right)^2 + \tfrac{1}{2}\left(v_3(t) - v^*\right)^2 + \tfrac{1}{2}\left(V(h_3(t)) - v_3(t)\right)^2 ,\quad \forall t \geq 0.
\end{equation}
We recall from \cref{cor_well_posed_N_cars} that $\inf_{t \geq 0} h_3(t) \geq h_{3,\min} > 0 $. A straightforward computation analogous to those presented in the proof of \cref{thm: main2} shows that
\begin{equation}\label{eq:evolution_F3}
\tfrac{\dd}{\dd t} \mathcal{F}_3(t) + p_3(t) \mathcal{F}_3(t) \leq q_3(t),
\end{equation}
in which
\begin{align*}
-p_3(t) &= \max\left\{-V'(h_3(t)), -\tfrac{2\beta}{h^2_3(t)}+\left|\alpha - V'(h_3(t)) + \tfrac{\beta}{h^2_3(t)}\right|, \right.\\
& \left.\quad -2 \alpha + V'(h_3(t)) + \left|\alpha - V'(h_3(t)) + \tfrac{\beta}{h^2_3(t)}\right|\right\}  \\
&= \max\left\{-V'(h_3(t)),~ \alpha - V'(h_3(t)) - \tfrac{\beta}{h^2_3(t)},~ -\alpha + \tfrac{\beta}{h^2_3(t)}\right\} \\
&= -\left(\alpha - \tfrac{\beta}{h^2_3(t)}\right) \leq -\left(\alpha - \tfrac{\beta}{h^2_{3,\min}}\right)
\end{align*}
and 
\[
q_3(t) = \left(\left(2 V(h_3(t)) - v^* - v_3(t)\right) V'(h_3(t)) + \tfrac{\beta}{h^2_3(t)} \left(2 v_3(t) - v^* - V(h_3(t))\right) \right)\sigma_2(t).
\]
Now, by Gr\"{o}nwall's inequality, we obtain
\begin{equation}\label{eq:F3_estimate}
\mathcal{F}_3(t) \leq \tfrac{\int_0^t \exp\left(\int_0^s p_3(\tau) \dd \tau\right) q_3(s) \dd s}{\exp\left(\int_0^t p_3(s) \dd s\right)} + \mathcal{F}_3(0) \exp\left(-\int_0^t p_3(s) \dd s\right).
\end{equation}
Since $p_3(t) \geq \alpha - \tfrac{\beta}{h^2_{3,\min}} > 0$ for all $t \geq 0$, we have $\exp\left(-\int_0^t p_3(s) \dd s\right) \xrightarrow{t \to \infty} 0$. For the first term on the right-hand side of \eqref{eq:F3_estimate}, we split the proof into the following cases:
\begin{enumerate}[label=(\roman*)]
  \item If the limit $\lim \limits_{t \to \infty} \int_0^t \exp\left(\int_0^s p_3(\tau) \dd \tau\right) q_3(s) \dd s$ does not blow up to $\pm \infty$, we can see that \[\tfrac{\int_0^t \exp\left(\int_0^s p_3(\tau) \dd \tau\right) q_3(s) \dd s}{\exp\left(\int_0^t p_3(s) \dd s\right)} \xrightarrow{t \to \infty} 0.\]
  \item If the limit $\lim \limits_{t \to \infty} \int_0^t \exp\left\{\int_0^s p_3(\tau) \dd \tau\right\} q_3(s) \dd s$ blows up to $\pm \infty$, then applying l'H\^opital's rule gives
      \begin{equation*}
      \lim_{t \to \infty} \tfrac{\int_0^t \exp\left(\int_0^s p_3(\tau) \dd \tau\right) q_3(s) \dd s}{\exp\left(\int_0^t p_3(s) \dd s\right)} = \lim_{t \to \infty} \tfrac{\exp\left(\int_0^t p_3(s) \dd s\right) q_3(t)}{\exp\left(\int_0^t p_3(s) \dd s\right) p_3(t)} = \lim_{t \to \infty} \tfrac{q_3(t)}{p_3(t)} = 0,
      \end{equation*}
      in which the last identity follows from the observations that $p_3(t) \geq \alpha - \tfrac{\beta}{h^2_{3,\min}} > 0$ for all $t \geq 0$ and $q_3(t) \xrightarrow{t \to \infty} 0$.
\end{enumerate}
Assembling these estimates together, we conclude that $\mathcal{F}_3(t) \xrightarrow{t \to \infty} 0$, from which the required convergence guarantees
\begin{equation*}
\lim\limits_{t \to \infty} v_3(t) = v^* ~\text{ and }~ \lim\limits_{t \to \infty} h_3(t) = \lim\limits_{t \to \infty} \left(x_2(t) - x_3(t) - l\right) = h^{*}.
\end{equation*}
follow immediately.
\end{proof}

\begin{remark}
\label{rem:generalization}
    Analogously to \cref{rem_well_posed_inifitely_cars}, the result of \cref{thm:generalization} can be generalized to the case for infinitely many vehicles. 
\end{remark}
\section{Conclusion and Future Work}
\label{sec: conclusion}
In this work, we established the well-posedness of the Bando-FtL model on an infinite time horizon by proving a uniform positive lower bound on the space headway of all followers, provided that the leader maintains a nonnegative velocity. We further investigated the model’s long-term behavior under the assumption that the leader travels at a constant speed. Our analysis also applies to configurations with infinitely many vehicles. As a next step, we seek to study the long-term behavior of the Bando-FtL model under weaker parameter assumptions. We also plan to examine the system’s asymptotic behavior when the leader's velocity undergoes oscillations rather than remaining constant. Ultimately, we aim to connect our results to equilibrium states or traveling-wave solutions to macroscopic traffic flow models obtained as many-particle limits of second-order microscopic models. 

\section*{Acknowledgements} 
This work was supported by the National Science Foundation under grant DMS-2418971.

\bibliographystyle{siamplain}
\bibliography{references}

\end{document}

%% file: definition.tex
\newcommand{\R}{\mathbb R}
\newcommand{\N}{\mathbb N}

\renewcommand{\phi}{\varphi}

\newcommand{\dd}{\ensuremath{\,\mathrm{d}}}

\renewcommand{\epsilon}{\varepsilon}



%% file: ex_shared.tex
\usepackage{lipsum}
\usepackage{amsfonts}
\usepackage{graphicx}
\usepackage{epstopdf}
\usepackage{algorithmic}
\ifpdf
  \DeclareGraphicsExtensions{.eps,.pdf,.png,.jpg}
\else
  \DeclareGraphicsExtensions{.eps}
\fi

\usepackage{enumitem}
\setlist[enumerate]{leftmargin=.5in}
\setlist[itemize]{leftmargin=.5in}

\newsiamremark{remark}{Remark}
\newsiamremark{hypothesis}{Hypothesis}
\crefname{hypothesis}{Hypothesis}{Hypotheses}
\newsiamthm{claim}{Claim}
\newsiamremark{fact}{Fact}
\crefname{fact}{Fact}{Facts}

\headers{long-term behavior of the ``Bando--follow-the-leader'' model}{F. Cao, X. Gong, and A. Keimer}

\title{An Example Article\thanks{Submitted to the editors DATE.
\funding{This work was funded by the Fog Research Institute under contract no.~FRI-454.}}}

\author{Dianne Doe\thanks{Imagination Corp., Chicago, IL 
  (\email{ddoe@imag.com}, \url{http://www.imag.com/\string~ddoe/}).}
\and Paul T. Frank\thanks{Department of Applied Mathematics, Fictional University, Boise, ID 
  (\email{ptfrank@fictional.edu}, \email{jesmith@fictional.edu}).}
\and Jane E. Smith\footnotemark[3]}

\usepackage{amsopn}
